\documentclass[letterpaper,12pt,oneside,reqno]{amsart}
\usepackage[utf8]{inputenc}%
\usepackage[english]{babel}%
\usepackage{amsmath,amssymb,amsthm,amsfonts}%
\usepackage{hyperref}%
\usepackage{enumerate}%
\usepackage{graphicx}
\usepackage[mathscr]{euscript}
\usepackage[DIV14]{typearea}
\usepackage[width=.9\textwidth]{caption}
\allowdisplaybreaks%
\numberwithin{equation}{section}%

\newcommand{\Z}{\mathbb{Z}}
\newcommand{\C}{\mathbb{C}}
\newcommand{\R}{\mathbb{R}}
\DeclareMathOperator{\E}{\mathbb{E}}
\renewcommand{\i}{\mathbf{i}}

\newcommand{\al}{\alpha}

\newcommand{\la}{\lambda}

\newcommand{\x}{\mathbf{x}}
\newcommand{\y}{\mathbf{y}}
\newcommand{\n}{\mathbf{n}}
\newcommand{\av}{\mathbf{a}}

\newcommand{\gen}{\mathscr{L}}
\newcommand{\RR}{\mathsf{R}}
\newcommand{\LL}{\mathsf{L}}

\newcommand{\gap}[1]{\ensuremath{\mathrm{gap}_{#1}}}
\newcommand{\W}{\mathbb{W}}
\DeclareMathOperator{\prob}{\mathrm{Prob}} 
\newcommand{\integral}{cumulative}

\newtheorem{proposition}{Proposition}[section]
\newtheorem{lemma}[proposition]{Lemma}
\newtheorem{corollary}[proposition]{Corollary}
\newtheorem{theorem}[proposition]{Theorem}
\newtheorem{claim}[proposition]{Conjecture}
\theoremstyle{definition}
\newtheorem{definition}[proposition]{Definition}
\newtheorem{remark}[proposition]{Remark}

\begin{document}
\title[The $q$-P\lowercase{ush}ASEP]{The $q$-P\lowercase{ush}ASEP: a new integrable model\\ for traffic in 1+1 dimension}

\author[I. Corwin]{Ivan Corwin}
\address{I. Corwin, Columbia University,
Department of Mathematics,
2990 Broadway,
New York, NY 10027, USA,
and Clay Mathematics Institute, 10 Memorial Blvd. Suite 902, Providence, RI 02903, USA,
and Massachusetts Institute of Technology,
Department of Mathematics,
77 Massachusetts Avenue, Cambridge, MA 02139-4307, USA}
\email{icorwin@mit.edu}

\author[L. Petrov]{Leonid Petrov}
\address{L. Petrov,
Department of Mathematics, Northeastern University, 360 Huntington ave., Boston, MA 02115, USA\newline
Institute for Information Transmission Problems, Bolshoy Karetny per. 19, Moscow, 127994, Russia}
\email{lenia.petrov@gmail.com}

\date{}

\maketitle

\begin{abstract}
	We introduce a new interacting (stochastic) particle system 
	$q$-PushASEP
	which interpolates between the
	\mbox{$q$-TASEP} of \cite{BorodinCorwin2011Macdonald}
	(see also \cite{BorodinCorwinSasamoto2012}, 
	\cite{BorodinCorwin2013discrete}, \cite{OConnellPei2012}, \cite{BorodinCorwinPetrovSasamoto2013})
	and the $q$-PushTASEP introduced recently 
	\cite{BorodinPetrov2013NN}. 
	In the $q$-PushASEP, particles 
	can jump to the left or to the right, and 
	there is a certain partially asymmetric pushing mechanism
	present. This particle system 
	has a nice interpretation
	as a model of 
	traffic on a one-lane highway.

	Using the quantum many body system approach, we explicitly compute the expectations of a large family
	of observables for this system
	in terms of nested contour integrals.
	We also discuss relevant Fredholm determinantal formulas
	for the distribution of the location of each particle,
	and connections of the model
	with a certain two-sided version of Macdonald processes
	and with the semi-discrete stochastic heat equation.
\end{abstract}

\section{Introduction and main results} 
\label{sec:introduction}

\subsection{Definition of the process} 
\label{sub:definition_of_the_process}

The $N$-particle \emph{$q$-PushASEP} 
($q$-deformed pushing asymmetric simple exclusion process)
is a continuous-time interacting particle system with
the state space consisting of 
ordered configurations 
$x_1>x_2>\ldots>x_N$, $x_i\in\Z$ (we assume that $N\ge1$ is fixed).
For convenience, we add two ``virtual'' particles 
$x_0=+\infty$ and $x_{N+1}=-\infty$,
and denote the state of the system
as
\begin{align}\label{space_X}
	X^N:=\Big\{\x=(-\infty=x_{N+1}<x_N<\ldots<x_2<x_1<x_0=+\infty)\colon 
	x_1,\ldots,x_N\in\Z\Big\}.
\end{align}
Let us also denote by
$\gap{i}:=x_{i-1}-x_i-1$ the 
$i$th gap between the particles.
{Throughout the paper, 
$q$ is a parameter belonging to
$(0,1)$.}

The dynamics of $q$-PushASEP 
$\{\x(t)\}_{t\ge0}$ depend on 
positive
parameters 
$a_1,\ldots,a_N$ and also on $\RR,\LL\ge0$ such that $\RR$ and $\LL$ are not simultaneously zero. 
It is described as follows (see Figure \ref{fig:qpushASEP}):
\begin{enumerate}[$\bullet$]
	\item (\emph{right jumps}) 
	Each particle $x_i(t)$, $1\le i\le N$, 
	jumps to the right by one 
	(i.e., instantaneously moves to position $x_i(t)+1$)
	at rate $a_i\RR\big(1-q^{\gap i(t)}\big)$, 
	independently of other particles.
	The jump rate of $x_i(t)$ vanishes if $\gap i(t)=0$, 
	which means that a particle cannot jump
	onto a site which is already occupied 
	(this is the \emph{exclusion mechanism}).

	\item (\emph{left jumps}) Each particle $x_i(t)$, $1\le i\le N$,
	jumps to the left by one
	(i.e., moves to position $x_i(t)-1$)
	at rate $a_i^{-1}\LL$,
	independently of other particles.
	There is also a \emph{mechanism of instantaneous pushes}
	present in left jumps.
	Namely, 
	if any particle
	$x_j(t)$ has moved to the left, 
	i.e., if $x_j(t+dt)=x_j(t)-1$, then
	$x_j(t)$ has a chance 
	to instantaneously (long-range)
	push its left neighbor 
	$x_{j+1}(t)$ to the left by one
	with probability $q^{\gap{j+1}(t)}$.
	If particle $x_{j+1}(t)$ is pushed then it 
	also has the possibility to 
	push its own left neighbor $x_{j+2}(t)$, and so on.
	When $\gap{j+1}(t)=0$, the probability of a push 
	becomes one, which means that a 
	particle moving to the left 
	always pushes
	a (possibly empty) cluster of
	its \emph{immediate} left neighbors.
\end{enumerate}
\begin{figure}[htbp]
	\begin{center}
		\includegraphics[scale=.38]{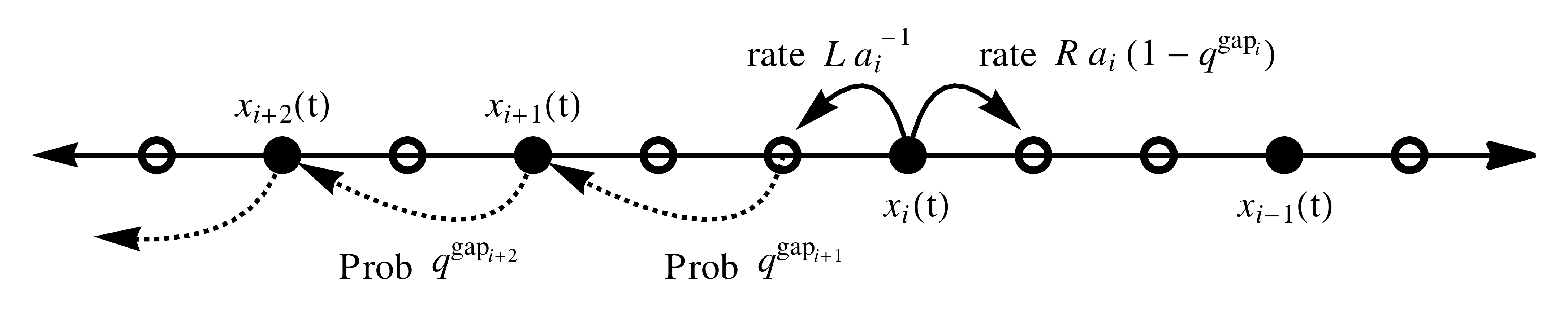}
	\end{center}  
  	\caption{Jump rates and probabilities
  	of pushes in $q$-PushASEP.}
  	\label{fig:qpushASEP}
\end{figure}

Clearly, the
$q$-PushASEP preserves the order of particles,
so we will always speak about the dynamics
of labeled particles
$x_1(t)> \ldots> x_N(t)$.
We assume that the $q$-PushASEP $\x(t)$ starts from the \emph{step initial condition}
defined as $x_i(0)=-i$, $i=1,\ldots,N$. 

It is worth noting that the 
first particle $x_1(t)$ performs a very simple dynamics:
it jumps to the right or to the left by one (independently
of other particles) at rates
$\RR a_1$ and $\LL a_1^{-1}$, respectively. Likewise, 
the first $n<N$ particles $x_1(t),\ldots,x_n(t)$
evolve without any dependence on those particles
$x_{n+1}(t),\ldots,x_N(t)$ 
to their left. Therefore, even though we have restricted
our attention to an $N$-particle system, 
we may also think of this as the evolution
of the $N$ rightmost particles in a system with more than $N$
particles.


\subsection{Traffic model} 
\label{sub:traffic_model}

The $q$-PushASEP may be viewed as a model of 
traffic on a one-lane highway in the following sense. 
Let $v\gg0$, and set
$c_j(t):=v t+x_j(t)$, $j=1,\ldots,N$,
where $x_1(t)> \ldots>x_N(t)$ evolve according to the $q$-PushASEP.
Viewing 
$c_1(t)> \ldots > c_N(t)$
as positions
of cars on the highway (i.e., we understand
their positions relative to a moving reference frame), one can interpret the dynamics as follows.

The jump of a particle $x_j$ to the right by one (under the $q$-PushASEP)
may be viewed as a brief 
acceleration of the car $c_j$,
after which $c_j$ becomes closer to $c_{j-1}$, and after that
continues to maintain the constant global speed $v$. 
Chances that $c_j$ will briefly accelerate are lower 
if the car $c_{j-1}$ is already close ahead
because of the rate 
$a_j\RR(1-q^{\gap j})$ of right jumps. 

The left jump of $x_j$ may be interpreted as a brief slowdown of the $j$th car,
after which it continues to maintain the constant global speed $v$.
When such a slowdown happens, the
car $c_{j+1}$ behind
$c_j$ sees the brake lights of $c_j$, and may also quickly slow down.
The probability of the latter event is higher when $c_{j+1}$ is closer to $c_j$
because of the pushing probability $q^{\gap {j+1}}$ in the $q$-PushASEP.
If $c_{j+1}$ decides to slow down, then $c_{j+2}$
in turn sees the brake lights of $c_{j+1}$, and may also decide to
brake, and so on.


\subsection{Relation to other models} 
\label{sub:relation_to_other_models}

When $\LL=0$
(i.e., only right jumps are allowed),
the $q$-PushASEP
turns into \emph{$q$-TASEP} ($q$-deformation of the totally
asymmetric simple exclusion process), which 
is an interacting
particle system introduced in 
\cite{BorodinCorwin2011Macdonald},
see also
\cite{BorodinCorwinSasamoto2012}, 
\cite{BorodinCorwin2013discrete},
\cite{BorodinCorwinPetrovSasamoto2013},
\cite{SasamotoWadati1998},
and 
\cite{OConnellPei2012}.
	
On the other hand, when 
$\RR=0$ (i.e., we permit only left jumps), 
our
process
essentially becomes the 
\emph{$q$-PushTASEP}
introduced in \cite{BorodinPetrov2013NN} as a 
one-dimensional marginal of a certain 
stochastic dynamics on two-dimensional
arrays of interlacing particles. 

Thus, the $q$-PushASEP interpolates between 
the $q$-TASEP and the $q$-PushTASEP.
See also 
Appendix \ref{sec:two_dimensional_dynamics} for an explanation of 
how the $q$-PushASEP is also related to a dynamics
on two-dimensional interlacing arrays.

Under the $q\to0$ degeneration, our process
becomes \emph{PushASEP} --- a two-sided particle system 
(in the sense that particles can jump to the left and to the 
right) 
which interpolates between TASEP and PushTASEP,
see \cite{alimohammadi1999two},
\cite{BorFerr08push}. 
The two latter processes appeared in 
\cite{Spitzer1970} 
(in that paper the PushTASEP was called the \emph{long-range TASEP}),
see also \cite{Liggett1985}, \cite{Liggett1999}.

See also 
\cite{Povolotsky_Mendes_2006},
\cite{Povolotsky2013}
for related developments.

\begin{remark}
	Similarly to \cite{BorFerr08push}, one can make the parameters
	$\RR$ and $\LL$ in the definition of the $q$-PushASEP
	depend on time (in a sufficiently nice way). 
	This will lead to replacement of 
	the quantities
	$\RR t$ and $\LL t$ 
	in our final formulas
	(e.g., \eqref{moments} or \eqref{Fred_det} below) by 
	$\int_{0}^{t}\RR(s)ds$
	and
	$\int_{0}^{t}\LL(s)ds$,
	respectively.
	To make exposition clearer, we will consider
	only constant $\RR$ and $\LL$.
\end{remark}


\subsection{Moments} 
\label{sub:moments}

To formulate one of our main results, 
define the Weyl chamber (of type A) as
\begin{align}\label{Weyl_chamber}
	\W^{k,N}_{\ge0}:=\{\n=(n_1,\ldots,n_k)\in\Z^{k}\colon N\ge 
	n_1\ge \ldots \ge n_k\ge0\}.
\end{align}
We compute  
joint 
$q$-moments
(or $q$-exponential moments)
of positions 
of several particles under the $q$-PushASEP:
\begin{theorem}\label{thm:moments}
	For any $\n\in\W^{k,N}_{\ge0}$,
	\begin{align}&
		\label{moments}
		\E\left(
		\prod_{i=1}^{k}q^{x_{n_i}(t)+n_i}
		\right)
		=\frac{(-1)^{k}q^{k(k-1)/2}}{(2\pi\i)^{k}}
		\oint \cdots\oint
		\prod_{1\le A<B\le k}
		\frac{z_A-z_B}{z_A-qz_B}
		\prod_{j=1}^{k}
		\left(\prod_{i=1}^{n_j}
		\frac{a_i}{a_i-z_j}\right)
		\frac{\Pi_t(qz_j)}{\Pi_t(z_j)}
		\frac{dz_j}{z_j},
	\end{align}
	where 
	\begin{align}\label{Pi_t}
		\Pi_t(z):=e^{t(\RR z+\LL z^{-1})}.
	\end{align}
	Here $\x(t)$ denotes the $q$-PushASEP started
	from the step initial condition $\{x_i(0)=-i\}_{i=1}^{N}$.
	The contour for 
	$z_A$ contains $a_1,\ldots,a_N$ and 
	all of 
	the contours $\{qz_B\}_{B>A}$, but not zero (see Fig.~\ref{fig:nested}
	for an example of contours).
\end{theorem}
A simple argument bounding the $q$-PushASEP by Poisson processes
shows that the moments in the left-hand side of 
\eqref{moments} are indeed finite (see \S \ref{sub:finiteness_of_moments}).

\begin{figure}[htbp]
	\begin{center}
		\includegraphics[scale=1.2]{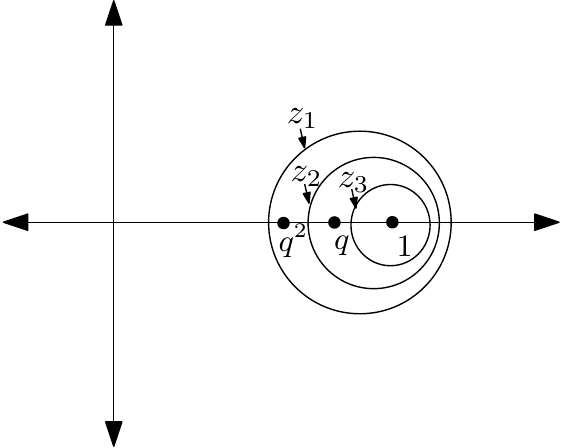}
	\end{center}  
  	\caption{Nested contours for $k=3$ and $a_i\equiv1$.}
  	\label{fig:nested}
\end{figure}

\begin{remark}
	It is worth noting that while the left-hand side of \eqref{moments}
	is symmetric in $n_1,\ldots,n_k$, the right-hand side is 
	\emph{not}. Theorem \ref{thm:moments}
	states the equality of the two expressions only for 
	$\n=(n_1,\ldots,n_k)$ belonging to the Weyl chamber.
\end{remark}


\subsection{True and free evolution equations} 
\label{sub:true_and_free_evolution_equations}

Our strategy of the proof of Theorem \ref{thm:moments} is the following.
We observe 
(see \S \ref{sec:true_evolution_equations}) that the expectations of 
$\prod_{i=1}^{k}q^{x_{n_i}(t)+n_i}$ for $\n\in\W^{k,N}_{\ge0}$
evolve according to closed systems of coupled ODEs,
which we call the \emph{true evolution equations}.
The equations' right-hand sides include as a summand 
the right-hand sides of
\cite[(3)]{BorodinCorwinSasamoto2012}
which corresponds to the $q$-TASEP,
and also new terms corresponding to the 
$q$-PushTASEP governing the left jumps.

Let us first recall 
\cite{BorodinCorwinSasamoto2012}
(see also \cite{BorodinCorwin2013discrete})
how one could solve the 
true evolution equations 
in the case
$\LL=0$ (i.e., when our particle system reduces to the $q$-TASEP).
For the $q$-TASEP,
the true evolution equations are 
constant coefficient and separable away from the boundary
of the Weyl chamber $\W^{k,N}_{\ge0}$
(but not on the boundary). 
In this case,
extending the constant coefficient, separable equations to all
of $\Z^{k}_{\ge0}$ results in the free evolution equations
on a function $u(t,\n)$, where $t\ge0$ and $\n\in\Z^{k}_{\ge0}$.
One of the results about the $q$-TASEP in \cite{BorodinCorwinSasamoto2012} 
is that the restriction to $\W^{k,N}_{\ge0}$
of a solution of the free evolution equations
satisfying certain boundary conditions
(resulting from the difference between the free and the true evolution
equations)
and with the right initial data in $\W^{k,N}_{\ge0}$
coincides with the solution of 
the true evolution equations.

In principle, there could be
a boundary condition 
for any possible combination of \emph{clusters}
(=~stings of equal coordinates)
in the vector $\n$.
A remarkable property of the $q$-TASEP
(\emph{integrability} in the language of (quantum) many body systems, 
cf. \cite{Bethe1931})
is that it suffices to consider only the following $k-1$ two-body boundary conditions:
for all $\n\in\Z^{k}_{\ge0}$
such that for some $i\in\{1,2,\ldots,k-1\}$
one has $n_i=n_{i+1}$,
\begin{align}\label{k-1_boundary}
	(\nabla_i-q\nabla_{i+1})u(t,\n)=0.
\end{align}
Here for a function $f\colon\Z\to\R$, we denote $(\nabla f)(n):=f(n-1)-f(n)$, and $\nabla_j$
above means that the difference operator acts in the $j$-th coordinate.

\medskip

Let us now explain how the 
$q$-PushASEP situation differs from that of the $q$-TASEP.
For $\LL>0$,
the corresponding true evolution
equations for any $\n\in\W^{k,N}_{\ge0}$
involve linear combinations 
of expectations of 
$\prod_{i=1}^{k}q^{x_{n_i}(t)+n_i}$ 
with $\n$ running
up to the boundary of $\W^{k,N}_{\ge0}$.
Thus, it is not a priori clear how to 
write down the free evolution
equations (in $\Z^{k}_{\ge0}$) for the $q$-PushASEP
such that 
their solutions satisfying the same $k-1$ boundary conditions
\eqref{k-1_boundary}
coincide with solutions of the true evolution equations
(in $\W^{k,N}_{\ge0}$). 

A way to write down the free evolution equations which we employ instead is to introduce 
\emph{another} set of $k-1$ conditions which we call \emph{\integral{}}.
For simpler notation, assume now that $a_i=1$ for all $i=1,\ldots,N$
(see \S \ref{sec:free_q_pushasep} for a general case).
The \integral{} conditions are the following: 
for all $\n\in\Z^{k}_{\ge0}$
such that for some $i\in\{1,2,\ldots,k-1\}$
one has $n_i=n_{i+1}$,
\begin{align}\label{k-1_integral}
	(\nabla_i^{-1}-q^{-1}\nabla_{i+1}^{-1})u(t,\n)=0.
\end{align}
Here by $\nabla^{-1}$ we mean the operator acting on 
$f\colon\Z\to\R$ as 
$(\nabla^{-1}f)(n):=-f(n)-f(n-1)- \ldots-f(1)$. 
Note that
$(\nabla\nabla^{-1} f)(n)=f(n)$,
but
$(\nabla^{-1}\nabla f)(n)=f(n)-f(0)$.
As before, 
$\nabla^{-1}_{j}$ means the application of the 
operator in the $j$th coordinate. 

We then obtain the \emph{free evolution equations}
for the $q$-PushASEP 
which are constant coefficient and separable in 
$\Z^{k}_{\ge0}$,
and prove that solutions of the free evolution equations 
satisfying
\eqref{k-1_boundary}--\eqref{k-1_integral}
and having the right initial data inside $\W^{k,N}_{\ge0}$
coincide with solutions of the true evolution equations 
for the $q$-PushASEP. 

The emergence of the \integral{} conditions
\eqref{k-1_integral} which might seem somewhat mysterious 
from the Bethe ansatz point of view (cf. the treatment 
of the $q$-TASEP in \cite{BorodinCorwinPetrovSasamoto2013})
appeared due to a certain 
``symmetry'' of formulas responsible for the
right ($q$-TASEP) and left ($q$-PushTASEP) jumps. We plan
to investigate deeper reasons behind 
these \integral{} conditions in a future work.


\subsection{Solving evolution equations for the $q$-PushASEP} 
\label{sub:solving_free_and_true_evolution_equations}

One readily sees that there exists a general class of solutions to the 
free evolution equations for the $q$-PushASEP, but it is not
immediately clear how one should combine them
in the right way
so that they satisfy \eqref{k-1_boundary}--\eqref{k-1_integral}.
However, when the $q$-PushASEP starts from the 
step initial configuration, it is possible to 
\emph{check} that the nested contour integral
expression in the right-hand side of 
\eqref{moments} satisfies the free evolution
equations, $k-1$ boundary and $k-1$ \integral{} conditions \eqref{k-1_boundary}--\eqref{k-1_integral}, 
and has the right initial data,
thus producing 
the desired moment formula.

The moment formula \eqref{moments} for $q$-TASEP 
was initially proved for all $n_i=n$, $i=1,\ldots,k$, from the Macdonald process 
framework of \cite{BorodinCorwin2011Macdonald}. 
The general $\n$ formula was guessed and 
checked in \cite{BorodinCorwinSasamoto2012} 
via the many body system approach, 
and reproved in the Macdonald process framework in \cite{BCGS2013}.
Our formula \eqref{moments} for the $q$-PushASEP
differs only in the function $\Pi_t(z)$
which was equal to $e^{tz}$ for the $q$-TASEP. 
Discrete-time $q$-TASEPs of \cite{BorodinCorwin2013discrete}
also admit nested contour integral formulas for moments
with other choices of $\Pi_t(z)$ 
(in \cite[Thm. 2.1]{BorodinCorwin2013discrete}
these functions are denoted by $f_t^{\mathrm{geo}}(z)$
and $f_t^{\mathrm{Ber}}(z)$). 
The concrete form \eqref{Pi_t} of $\Pi_t(z)$
for the $q$-PushASEP can be guessed from any of the three sources:
\begin{enumerate}[(1)]
	\item Applying the nested contour integral ansatz 
	for solving the free evolution equations.
	\item By analogy with the 
	PushASEP (i.e., the $q=0$ case),
	e.g., see \cite[Prop. 2.1]{BorFerr08push}.
	Presence of factors of the form 
	$e^{t(\RR z+\LL z^{-1})}$
	in the PushASEP is a manifestation of its connection to the
	algebraic framework of the
	\emph{two-sided} Schur processes \cite{Borodin2010Schur}. 
	\item 
	The $q$-PushASEP 
	should fit into a more general framework of
	the \emph{two-sided
	Macdonald processes}
	extending the theory of \cite{BorodinCorwin2011Macdonald},
	\cite{BCGS2013}.
	The present paper provides a motivation 
	for a further investigation of the two-sided Macdonald processes.
	See also Appendix \ref{sec:two_dimensional_dynamics}.
\end{enumerate}


\subsection{Fredholm determinant} 
\label{sub:fredholm_determinant}

If $\LL>0$, observables of the form $\E\left(\prod_{i=1}^{k}q^{x_{n_i}(t)+n_i}\right)$
grow rapidly in $k$, namely, as $c_1\exp\left\{c_2 e^{c_3k}\right\}$ (for suitable positive constants). 
This means that the moments are not sufficient to identify the distribution 
of the process (at any given positive time). 

However, 
using \eqref{moments}
and the rigorously proved result for 
$\LL=0$, one can \emph{formally} write down a conjectural Fredholm 
determinantal formula for the 
$q$-Laplace transform of $q^{x_n(t)+n}$ (for any $1\le n\le N$).
For simplicity, assume that all $a_i\equiv 1$.
We will use the notation 
\begin{align*}
	(a;q)_{\infty}:=\prod_{i=0}^{\infty}(1-aq^{i}),
	\qquad
	(a;q)_{k}:=\prod_{i=0}^{k-1}(1-aq^{i}).
\end{align*}

\begin{claim}\label{claim:Fredholm}
	For all $\zeta\in\C\setminus\R_{>0}$,
	\begin{align}\label{Fred_det_identity}
		\E\left(\frac{1}{(\zeta q^{x_n(t)+n};q)_{\infty}}\right)
		=\det(I+K_{\zeta}).
	\end{align}
	Here $\det(I+K_{\zeta})$ is the Fredholm determinant
	of $K_{\zeta}\colon L^{2}(C_1)\to L^{2}(C_1)$,
	where $C_1$ is a small positively 
	oriented circle containing 1, and $K_{\zeta}$ 
	is an integral operator with kernel
	\begin{align}\label{Fred_det}
		K_{\zeta}(w,w')=\frac{1}{2\pi\i}
		\int_{-\i\infty+1/2}^{{\i\infty+1/2}}
		\frac{\pi}{\sin(-\pi s)}(-\zeta)^{s}
		\frac{G(q^{s}w)}{G(w)}
		\frac{1}{q^{s}w-w'}ds,
	\end{align}
	with (see \eqref{Pi_t})
	\begin{align*}
		G(w):=(w;q)_{\infty}^{n}\Pi_t(w).
	\end{align*}
\end{claim}
A formal approach to establish \eqref{Fred_det_identity} is to expand 
$\E\left({1}/{(\zeta q^{x_n(t)+n};q)_{\infty}}\right)$
by means of the $q$-Binomial theorem, and 
interchange the expectation and the summation
in the resulting series. 
A general scheme 
of doing this is explained in \cite[\S3]{BorodinCorwinSasamoto2012}.
However, for $\LL>0$, our moments of the $q$-PushASEP lead to a \emph{divergent} series
after one interchanges the expectation and the summation.

This is quite similar to the issue which arises in the polymer replica method, in which 
one attempts to recover the Laplace transform of the solution 
to the stochastic heat equation from a divergent moment generating series 
\cite{Dotsenko},
\cite{Calabrese_LeDoussal_Rosso}. 
We believe that for the $q$-PushASEP
this issue of divergence can be resolved 
(and thus \eqref{Fred_det} can be rigorously justified)
by passing to a suitable discrete-time approximation
(one may call it \emph{regularization})
possessing nested contour integral formulas 
similar to those of Theorem \ref{thm:moments}.
In this approximation, the derivation of a Fredholm determinantal
formula would be rigorous, and then a rather straightforward
continuous-time limit would yield the proof of Conjecture \ref{claim:Fredholm}. Constructing suitable discrete-time approximations 
is the subject of a future work \cite{MatveevPetrov2014}.

For $\LL=0$, the Fredholm determinantal formula \eqref{Fred_det} 
corresponds to the $q$-TASEP
and a proof of the conjecture 
appeared in \cite{BorodinCorwin2011Macdonald}, 
see also \cite{BorodinCorwinSasamoto2012}.
It was established by 
interchanging the expectation and the summation,
which is perfectly valid in this case.
Indeed, for $\LL=0$ (and the step initial configuration), 
all coordinates $x_n(t)+n$ are nonnegative. Thus, the expectations 
$\E (q^{k(x_n(t)+n)})$ are all bounded by one, and thus the 
series $\sum_{k=0}^{\infty}
{\zeta^{k}\E (q^{k(x_n(t)+n)})}/{(q;q)_{k}}$
is convergent for small enough
values of $\zeta$.


\subsection{Outline} 
\label{sub:outline}

In \S \ref{sec:true_evolution_equations} we discuss the 
$q$-PushASEP in detail, and write down the true evolution equations
for the observables in the left-hand side of \eqref{moments}.
We also suggest a Markov process dual to the $q$-PushASEP.
In \S \ref{sec:free_q_pushasep} we reduce the true evolution
equations to the free evolution equations with 
$k-1$ boundary and $k-1$ \integral{} conditions. We show that a solution
of the free evolution equations also satisfies the true evolution equations.
In \S \ref{sec:nested_contour_integral_formulas_for_the_q_pushasep}
we check that the nested contour integral
formula in the right-hand side of \eqref{moments}
satisfies the free evolution equations, and thus prove 
Theorem \ref{thm:moments}. We also discuss the Fredholm determinantal
formula (Conjecture \ref{claim:Fredholm}).
In Appendix \ref{sec:two_dimensional_dynamics}
we describe how the $q$-PushASEP is 
related to (two-sided) Macdonald processes.
In Appendix \ref{sec:formal_scaling_limit_as_}
we briefly discuss connections of 
our model with the semi-discrete stochastic heat equation. 


\subsection{Acknowledgments} 
\label{sub:acknowledgments}

The authors would like to thank 
Alexei Borodin
for very helpful discussions and remarks. 
IC was partially supported by the NSF through DMS-1208998 as well as by Microsoft Research through the Schramm Memorial Fellowship, and by the Clay Mathematics Institute through a Clay Research Fellowship.
LP~was partially supported by 
the RFBR-CNRS grant 11-01-93105.



\section{True evolution equations} 
\label{sec:true_evolution_equations}

In this section we write down closed systems
of coupled ODEs ({true evolution equations}) which are satisfied by
the expectations of the observables
of the form
$\prod_{i=1}^{k}q^{x_{n_i}(t)+n_i}$, 
where
$\n=(n_1,\ldots,n_k)$
belongs to the Weyl chamber
$\W^{k,N}_{\ge0}$ \eqref{Weyl_chamber}.

\subsection{Finiteness of moments} 
\label{sub:finiteness_of_moments}

\begin{lemma}\label{lemma:moments_finite}
	Let $\x(t)$ be the position at time $t\ge0$ 
	of the $q$-PushASEP started from any 
	fixed initial condition, i.e., from any point of $X^{N}$ 
	defined in \eqref{space_X}.
	Then for any $\n\in\W^{k,N}_{\ge0}$, the expectation 
	$\E\left(\prod_{i=1}^{k}q^{x_{n_i}(t)+n_i}\right)$ is finite.
\end{lemma}
\begin{proof}
	Left jumps of the $q$-PushASEP 
	introduce factors of $q^{-1}$ into
	$\E\left(\prod_{i=1}^{k}q^{x_{n_i}(t)+n_i}\right)$, and right
	jumps lead to factors of $q$. Since $0<q<1$, we need to estimate only the left jumps.

	Observe that the leftmost particle $x_N(t)$ has the possibility to be pushed
	to the left by any of the particles, so it can go to the left as far as a Poisson process 
	with rate $\LL (a_1^{-1}+\ldots+a_N^{-1})$. 
	Since the Poisson distribution has finite exponential moments 
	(i.e., $\E(e^{z\xi})<\infty$ for all $z\in\C$, where $\xi$ has Poisson distribution),
	we see that the claim holds.
\end{proof}


\subsection{Markov generator of $q$-PushASEP} 
\label{sub:markov_generator_of_q_pushasep}

It is readily seen from the definition 
(\S \ref{sub:definition_of_the_process})
that 
the Markov generator of the $q$-PushASEP
(acting on functions $f\colon X^N\to\R$)
has the form
\begin{align}&
	(\gen^{\text{$q$-PushASEP}}
	f)(\x)
	=\sum_{i=1}^{N}\RR a_i
	\big(1-q^{x_{i-1}-x_i-1}\big)
	\big(f(\x^{+}_{i})-f(\x)\big)
	\label{gen_qpush}
	\\&\hspace{145pt}+
	\sum_{i=1}^{N}
	\LL a_i^{-1}
	\sum_{j=i}^{N}q^{x_i-x_j-(j-i)}
	\big(1-q^{x_j-x_{j+1}-1}\big)
	\big(f(\x^{-}_{j,i})-f(\x)\big).
	\nonumber
\end{align}
Here we have denoted for all $1\le i\le j\le N$:
\begin{align*}
	\x_i^+&:=(x_1,\ldots,x_{i-1},x_i+1,x_{i+1},\ldots,x_N);
	\\
	\x^{-}_{j,i}&:=(x_1,\ldots,x_{i-1},x_{i}-1,x_{i+1}-1,\ldots,x_{j-1}-1,x_{j}-1,x_{j+1},\ldots,x_N).
\end{align*}
That is, $\x_i^+$ corresponds to the configuration
in which the $i$th particle has jumped to the right by one, and 
$\x^{-}_{j,i}$ means the configuration in which the particles 
with indices $m=i,i+1,\ldots,j$ have jumped to the left by one.
Note that if any of these jumps breaks the strict order of the particles, 
then the coefficient in \eqref{gen_qpush} by the corresponding term vanishes. 
This reflects the fact that the $q$-PushASEP preserves the order of the particles.


\subsection{Remark: Stationary distributions} 
\label{sub:remark_stationary_distributions}

Here let us present a calculation which suggests
how stationary measures of the $q$-PushASEP with infinitely many particles 
$-\infty<\ldots<x_1<x_0<x_{-1}<\ldots<+\infty$ look like (without discussing the 
existence of this process or proving that these measures are indeed stationary).
Assume translation invariance, i.e., that $a_i=1$ for all $i\in\Z$.

The case of the $q$-TASEP
(i.e., when $\LL=0$)
is discussed in 
\cite[\S 3.3.3]{BorodinCorwin2011Macdonald}.
There
the stationary measures 
are those for which the \emph{gaps}
$x_{i-1}-x_i-1=\gap{i}$ between particles are independent and
have the \emph{$q$-geometric distribution}
$\mathrm{qGeo}\left(\al\RR^{-1}\right)$, where $\al\in[0,\RR)$ is arbitrary:
\begin{align*}
	\prob\big(x_{i-1}-x_i-1=k\big)=(\al\RR^{-1};q)_{\infty}\frac{(\al\RR^{-1})^{k}}{(q;q)_{k}},\qquad
	k=0,1,\ldots.
\end{align*}

One can perform a formal calculation suggesting that
this distribution is also stationary for the 
$q$-PushTASEP part of the dynamics. Indeed, 
during a small time interval $dt$, each
$\gap{i}$ can  
increase by one with probability $\LL\, dt$ (which corresponds to $x_i$ jumping to the left).
Next, observe that the particle $x_{i-1}$ moves to the left at total rate
(accounting for all possible pushes that $x_{i-1}$ can receive from the left)
\begin{align}\label{sum_LRal}
	\LL\left(1+(1-\al\RR^{-1})+(1-\al\RR^{-1})^2+\ldots\right)=\frac{\LL\RR}{\al},
\end{align}
because $1-\al\RR^{-1}=\E (q^{\gap{i}})$ for all $i\in\Z$ (which readily follows from the 
$q$-Binomial theorem). This means that 
during a small time interval $dt$, the value of $\gap{i}$ 
can decrease by one with probability $\frac{\LL\RR}{\al}(1-q^{\gap{i}})dt$.
Here the factor $1-q^{\gap{i}}$ in the latter expression is the probability that the moved 
particle $x_{i-1}$ did \emph{not} push $x_{i}$. 
One can readily check that 
the law $\gap{i}\sim \mathrm{qGeo}\left(\al\RR^{-1}\right)$
is invariant for the 
one-dimensional Markov chain on $\Z_{\ge0}$ which we have just described.
This suggests
that this law should be preserved
by the $q$-PushASEP evolution.\footnote{One needs to additionally 
justify that $\gap{i}$ indeed evolves according
to this one-dimensional Markov chain.}

In the non-translation invariant case 
i.e., when the $a_i$'s are different,\footnote{One should also 
impose reasonable growth and decay assumptions on the $a_i$'s.} the consideration 
of the right jumps (i.e., the $q$-TASEP dynamics) leads to the following distributions of the gaps: 
$\gap{i}\sim \mathrm{qGeo}\left(\al\RR^{-1}a_i^{-1}\right)$.
Then the 
series in \eqref{sum_LRal} is no longer a geometric progression, 
but it still sums to $\frac{\LL\RR}{\al}$, which suggests that the 
independent
$q$-geometric
gaps 
$\mathrm{qGeo}\left(\al\RR^{-1}a_i^{-1}\right)$
are preserved by the left ($q$-PushTASEP) jumps
in the non-translation invariant setting
as~well.

It would be interesting 
to generalize the 
coupling approach
of \cite{Balasz_Komjathy_Seppalainen} to the
two-sided setting.


\subsection{$\y$-variable notation} 
\label{sub:notation}

Our aim now is to understand how the generator
\eqref{gen_qpush} acts on moments $\E\left(\prod_{i=1}^{k}q^{x_{n_i}(t)+n_i}\right)$.
It is convenient to pass from the coordinates $\n\in\W^{k,N}_{\ge0}$ to
a new set of coordinates. Denote
\begin{align}\label{Y_space}
	Y^N:=\Big\{\y=(y_0,y_1,\ldots,y_N)\in\Z^{N+1}_{\ge0}\Big\},
	\qquad
	Y^N_k:=\Big\{\y\in Y^N\colon \sum_{i=0}^{N}y_i=k\Big\}.
\end{align}
To each $\n\in\W^{k,N}_{\ge0}$ associate $\y(\n)\in Y^{N}_{k}$
defined by $y_i(\n):=|\{j\colon n_j=i\}|$. 
In the reverse direction, for any $\y\in Y^{N}_{k}$,
denote by 
$\n(\y)$
the unique $\n\in\W^{k,N}_{\ge0}$ for which $\y(\n)=\y$.
To illustrate, if $\n=(5,5,4,2,1,1,1)$, then $\y(\n)=(0,3,1,0,1,2)$.
We will call the number of nonzero coordinates of $\y$ the number
of \emph{clusters} of $\n$ (the present example has four clusters).

Let us define, for each $x\in X^{N}$ and $\y\in Y^N$,
\begin{align}\label{H_function}
	H(\x,\y):=\prod_{i=0}^{N}q^{(x_i+i)y_i}.
\end{align}
The product above starts from zero
which means that, by agreement, $H(\x,\y)=0$ if $y_0>0$.


\subsection{Action of $\gen^{\text{\rm{}$q$-PushASEP}}$ on $H(\x,\y)$} 
\label{sub:action_of_gen_q-pushasep_on_h_x_y_}

For all $0\le i \le j\le N$ denote
\begin{align*}
	\y^{j,i}:=(y_0,y_1,\ldots,y_{i-1},y_i+1,y_{i+1},\ldots,
	y_{j-1},y_j-1,y_{j+1},\ldots,y_N).
\end{align*}
That is, in $\y^{j,i}$ the $j$th coordinate is decreased by one, and the
$i$th coordinate is increased by one.
Clearly, $\y^{i,i}=\y$.

Denote by $\gen^{\text{dual}}$ the following operator acting on functions 
$g\colon Y^N\to\R$:
\begin{align}\label{dual_operator}
	(\gen^{\text{\rm{}dual}}g)(\y):=
	\sum_{i=1}^{N}\RR a_i 
	(1-q^{y_i})\big(g(\y^{i,i-1})-g(\y)\big)
	+\sum_{i=1}^{N}\LL a_i^{-1}
	\sum_{j=i}^{N}(q^{-y_j}-1)q^{-y_i- \ldots-y_{j-1}}	
	g(\y^{j,i}).
\end{align}
\begin{remark}\label{rmk:TAZRP}
	Note that the first sum (containing the parameter $\RR$) is the
	Markov generator of the \emph{$q$-Boson particle system} (a certain 
	totally asymmetric 
	zero range process)
	which is dual to the $q$-TASEP, see \cite{BorodinCorwinSasamoto2012}
	(where this process was called $q$-TAZRP)
	and also \cite{BorodinCorwinPetrovSasamoto2013}. 
	The second summand is new and it is responsible for the 
	left jumps (which are governed by the $q$-PushTASEP, cf. 
	\S \ref{sub:relation_to_other_models}). 
	See also \S \ref{sub:dual_markov_process_to_q_pushasep} below.
\end{remark}

\begin{proposition}\label{prop:action_on_H}
	For any $\x\in X^N$ and $\y\in Y^N$ we have 
	\begin{align*}
		\gen^{\text{\rm{}$q$-PushASEP}}_{\x}H(\x,\y)
		=
		\gen^{\text{\rm{}dual}}_{\y}H(\x,\y),
	\end{align*}
	where the subscripts $\x$ and $\y$ in the operators mean the
	variables in which the operators act.
\end{proposition}
\begin{proof}
	This follows from the observations
	\begin{align*}
		H(\x^{+}_{i},\y)-
		H(\x,\y)&=
		(q^{y_i}-1)H(\x,\y);\\
		H(\x^{-}_{j,i},\y)-
		H(\x,\y)&=
		(q^{-y_i-y_{i+1}- \ldots-y_j}-1)H(\x,\y);\\
		(1-q^{x_{i-1}-x_{i}-1})H(\x,\y)&=
		H(\x,\y)-H(\x,\y^{i,i-1});\\
		q^{x_i-x_j-(j-i)}
		\big(1-q^{x_j-x_{j+1}-1}\big)
		H(\x,\y)&=
		H(\x,\y^{j,i})-
		H(\x,\y^{j+1,i}).
	\end{align*}
	To get \eqref{dual_operator} after applying
	the above identities to \eqref{gen_qpush},
	one should also regroup summands 
	in the second sum (which contains the parameter~$\LL$)
	by collecting the coefficients by each $g(\y^{j,i})$.
\end{proof}


\subsection{True evolution equations} 
\label{sub:true_evolution_equations_y}

Proposition \ref{prop:action_on_H} motivates the following definition:

\begin{definition}\label{def:true}
	A function $h(t,\y)$, $t\ge0$, $\y\in Y^N$, is said to satisfy 
	the \emph{true evolution equations}
	with initial conditions $h_0(\y)$ 
	if
	\begin{enumerate}[(1)]
		\item For all $\y\in Y^N$ and $t\ge0$:
		\begin{align}\label{true_ODE}
			\frac{d}{dt}h(t,\y)=\gen^{\text{dual}}h(t,\y),
		\end{align}
		where the operator $\gen^{\text{dual}}$ given by \eqref{dual_operator}
		acts in the variables $\y$.
		\item (boundary conditions)
		For all $\y\in Y^N$ such that $y_0>0$, $h(t,\y)\equiv 0$ for all $t\ge 0$.
		\item (initial conditions)
		For all $\y\in Y^N$, $h(0,\y)=h_0(\y)$.
	\end{enumerate}
\end{definition}
\begin{lemma}\label{lemma:unique_sol}
	The above true evolution equations have unique solutions.
\end{lemma}

\begin{proof}
	The proof is the same as for the $q$-TASEP, see 
	\cite[Lemma 3.5]{BorodinCorwinSasamoto2012}.

	The operator $\gen^{\text{dual}}$ maps the space 
	of functions $g\colon Y^N_k\to\R$ onto itself. 
	Therefore, the true evolution equations reduce 
	to a collection of finite closed systems of ODEs indexed 
	by $k\ge1$. 

	Moreover, for each fixed $k$, the system of the
	true evolution equations is \emph{triangular}. Namely, 
	the derivative $\frac{d}{dt}h(t,\y)$ depends only on those
	$h(t,\y')$ for which $y'_i+\ldots+y_N'\le y_i+\ldots+y_N$
	for all $0\le i\le N$. The existence and uniqueness of solutions 
	to each finite, closed, triangular system of linear ODEs 
	is justified by standard methods, e.g., see \cite{ODEs}.
\end{proof}

\begin{theorem}\label{thm:true}
	For any $\x\in X^N$, and for the $q$-PushASEP $\{\x(t)\}_{t\ge0}$ started 
	from an arbitrary (non-random) initial condition $\x(0)=\x$, 
	the function $h(t,\y):=\E(H(\x(t),\y))$ solves the 
	true evolution equations with initial data $h_0(\y)=H(\x,\y)$.
\end{theorem}
By linearity, one can also consider good enough
random initial configurations for the $q$-PushASEP.
In this case, one should take the initial data
to be 
$h_0(\y)=\E(H(\x,\y))$, where the expectation is with respect to the
initial configuration $\x$.
\begin{proof}
	Due to Lemma \ref{lemma:unique_sol}, it suffices to check that the 
	function $h(t,\y):=\E^{\x}\big(H(\x(t),\y)\big)$ (the superscript $\x$ means that the 
	expectation is taken with respect to the $q$-PushASEP starting from $\x$)
	satisfies \eqref{true_ODE} (boundary and initial conditions are straightforward).
	
	One has
	\begin{align*}
		\frac{d}{dt}\E^{\x}\big(H(\x(t),\y)\big)=
		\gen^{\text{$q$-PushASEP}}\E^{\x}\big(H(\x(t),\y)\big)
		=
		\E^{\x}\big(\gen^{\text{$q$-PushASEP}}H(\x(t),\y)\big).
	\end{align*}
	The first equality is the backwards Kolmogorov equations (essentially, the definition of a Markov
	generator), and the second one follows from the fact that the generator 
	$\gen^{\text{$q$-PushASEP}}$
	of the Markov semigroup of the $q$-PushASEP commutes 
	with the operators from this semigroup. 

	Next, using Proposition \ref{prop:action_on_H}, we can continue
	the above equalities ($\gen^{\text{$q$-PushASEP}}$
	and $\gen^{\text{dual}}$ act on $\x$ and $\y$ variables, respectively)
	\begin{align*}
		\E^{\x}\big(\gen^{\text{$q$-PushASEP}}H(\x(t),\y)\big)
		=
		\E^{\x}\big(\gen^{\text{dual}}H(\x(t),\y)\big)
		=
		\gen^{\text{dual}}\E^{\x}\big(H(\x(t),\y)\big).
	\end{align*}
	The last equality is due to the fact that the
	expectation is taken with respect to the $\x$ variables while the operator
	$\gen^{\text{dual}}$ acts in the $\y$ variables.
\end{proof}


\subsection{Remark: Markov process dual to the $q$-PushASEP} 
\label{sub:dual_markov_process_to_q_pushasep}

For $\LL>0$, the operator $\gen^{\text{dual}}$
\eqref{dual_operator} is \emph{not} a generator of any 
continuous-time Markov
process on the space $Y^N$ (cf. Remark \ref{rmk:TAZRP} about the $\LL=0$ case).
Indeed, applying this operator to the identity function, one has
\begin{align*}
	\gen^{\text{dual}}\mathbf{1}=
	\sum_{i=1}^{n}\LL a_i^{-1}(q^{-y_i- \ldots-y_N}):=C(\y).
\end{align*}
However, the fact that $C(\y)$ is not zero is the only obstacle 
preventing $\gen^{\text{dual}}$ from being a Markov generator.
Thus, let us define the following operator 
acting on functions $g\colon Y^N \to \R$ by
\begin{align}\label{dual_Markov}
	&(\gen^{\text{dual Markov}}
	g)(\y)
	:=
	(\gen^{\text{dual}}
	g)(\y)-C(\y)g(\y)
	\\&
	\hspace{10pt}=
	\sum_{i=1}^{N}\RR a_i 
	(1-q^{y_i})\big(g(\y^{i,i-1})-g(\y)\big)
	+\sum_{i=1}^{N}
	\LL a_i^{-1}
	\sum_{j=i+1}^{N}(q^{-y_j}-1)q^{-y_i- \ldots-y_{j-1}}	
	\big(g(\y^{j,i})-g(\y)\big).
	\nonumber
\end{align}
One readily sees that this operator can serve as a generator
of a continuous-time Markov process on $Y^N$; denote
this process by $\y(t)$. Representing the 
state space $Y^N$ as in Fig.~\ref{fig:qpush_dual_Markov}, 
we see that the transitions in $\y(t)$ look as follows:
\begin{enumerate}[(1)]
	\item ($q$-TASEP part) For each $i\in\{2,\ldots,N\}$,
	the coordinate $y_i(t)$ decreases by one and simultaneously
	$y_{i-1}(t)$ increases by one 
	(=~a particle jumps from site $i$ to site $i-1$)
	at rate $\RR a_i(1-q^{y_i(t)})$.
	\item ($q$-PushTASEP part) For each $1\le i<j\le N$, 
	the coordinate $y_j(t)$ decreases by one and simultaneously
	$y_{i}(t)$ increases by one 
	(=~a particle jumps from site $j$ to site $i$)
	at rate
	$\LL a_i^{-1}(q^{-y_j(t)}-1)q^{-y_i(t)- \ldots-y_{j-1}(t)}$.
\end{enumerate}
All these transitions occur independently.
Note that for $\LL>0$, the process $\y(t)$ is not 
zero range.
\begin{figure}[htbp]
	\vspace{-25pt}
	\begin{center}
		\includegraphics[scale=.24]{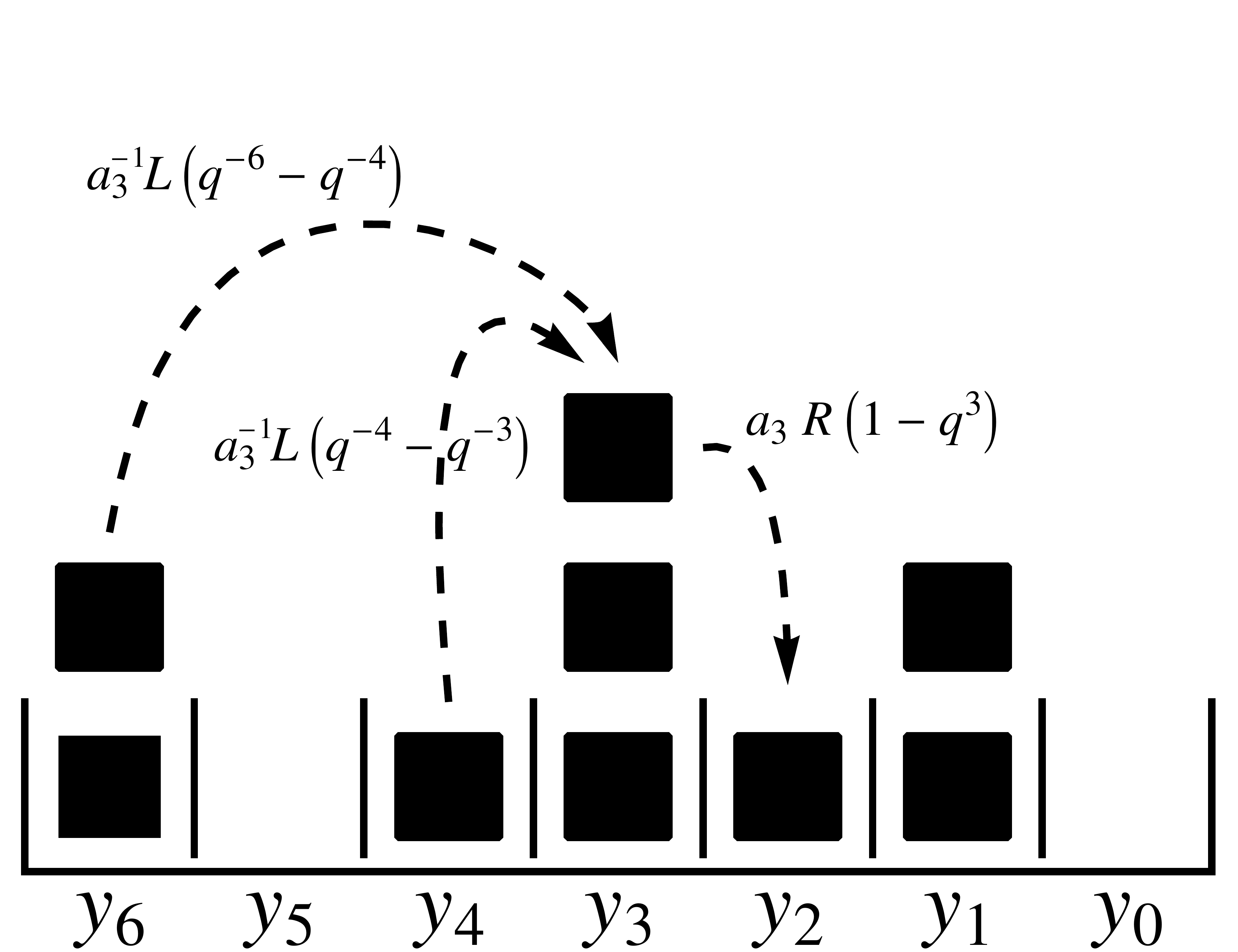}
	\end{center}  
  	\caption{Markov process $\y(t)$ dual to $q$-PushASEP.
  	Indicated are all possible jumps involving 
  	the parameter $a_3$.}
  	\label{fig:qpush_dual_Markov}
\end{figure}

Clearly, Proposition \ref{prop:action_on_H} implies that
the Markov generators of the $q$-PushASEP $\x(t)$ \eqref{gen_qpush}
and of the above process $\y(t)$ \eqref{dual_Markov}
satisfy the following generalized duality relation with respect
to the same function $H(\x,\y)$ \eqref{H_function}:
\begin{align*}
	\gen^{\text{\rm{}$q$-PushASEP}}_{\x}H(\x,\y)
	=
	\gen^{\text{\rm{}dual Markov}}_{\y}H(\x,\y)
	+C(\y)H(\x,\y).
\end{align*}
Consequently, the expectations of $H(\x,\y)$ with respect to 
evolution of 
the processes
$\x(t)$ and $\y(t)$ are related as
\begin{align}\label{generalized_duality}
	\E^{\x}\big(H(\x(t),\y)\big)=
	\E^{\y}\left(H(\x,\y(t))
	e^{\int_{0}^{t}C(\y(s))ds}\right).
\end{align}
Here in the left-hand side we have the expectation
under $\x(t)$ started from $\x$, and on the right there
is the expectation under the law of the process
$\y(t)$ started from $\y$. About (generalized) duality
of Markov processes, e.g., see
\cite[Ch. 4.4]{Ethier1986} and references therein.

One could use the generalized duality \eqref{generalized_duality}
to provide a probabilistic insight into Theorem \ref{thm:true}.
However, from the many body systems point of view
the process $\x(t)$ is not required to be dual to any 
Markov process. One only needs the fact that the observables 
$\E\left(\prod_{i=1}^{k}q^{x_{n_i}(t)+n_i}\right)$
evolve according to a closed system of ODEs.



\section{Free evolution equations with $k-1$ boundary and $k-1$ \integral{} conditions for the $q$-PushASEP} 
\label{sec:free_q_pushasep}

The goal of this section is to reduce the 
true evolution equations 
for the two-sided $q$-PushASEP
(Theorem \ref{thm:true}) 
to free evolution equations which are 
constant coefficient and separable (see the 
discussion in 
\S \ref{sub:true_and_free_evolution_equations}--\ref{sub:solving_free_and_true_evolution_equations}
for more detail). 


Let $\av:=(a_1,\ldots,a_N)$,
and recall that all $a_i$ are positive. 
Define the following operators acting on functions
$f\colon\Z\to\R$:
\begin{align*}
	(\nabla_{\av}^{\phantom{1}}f)(n):=a_n\big(f(n-1)-f(n)\big),\quad
	(\nabla_{\av}^{-1}f)(n):=-a_n^{-1}f(n)-a_{n-1}^{-1}f(n-1)-\ldots-a_1^{-1}f(1).
\end{align*}
By agreement, let us add ``dummy parameters'' $a_n$, $n>N$. They
do not enter main formulas of this section, but it
is convenient to include them to avoid the requirement that $n\le N$.
Equivalently, one may think of dealing with the process with 
infinitely many particles to the left of the origin (and finitely many
particles to the right of the origin), cf. the end of 
\S \ref{sub:definition_of_the_process}.

Clearly,
\begin{align*}
	(\nabla_\av^{\phantom{1}}\nabla_\av^{-1} f)(n)=f(n),
	\qquad
	(\nabla_\av^{-1}\nabla_\av^{\phantom{1}} f)(n)=f(n)-f(0).
\end{align*}
For a function on $\Z^{k}$, let $[\nabla_\av^{\phantom{1}}]^{\phantom{1}}_{j}$ and 
$[\nabla_\av^{-1}]^{\phantom{1}}_{j}$ denote the application 
of the corresponding operators in the $j$-th variable.



\begin{definition}\label{def:free}
	We say that a function 
	$u\colon\R_{\ge0}\times\Z_{\ge0}^{k}\to\R$ satisfies
	the free evolution equations with $k-1$ boundary conditions, $k-1$
	\integral{} conditions,
	and (partial) initial conditions $h_0$ inside the Weyl chamber 
	$\W^{k,N}_{\ge0}\subseteq\Z^{k}_{\ge0}$,
	if 
	\begin{enumerate}[(1)]
		\item 
		For all $\n\in\Z^{k}_{\ge0}$ and $t\ge0$,
		\begin{align}
			\frac{d}{dt}u(t,\n)=
			\RR\cdot
			(1-q)\sum_{i=1}^{k}
			[\nabla_\av^{\phantom{1}}]_{i}^{\phantom{1}}u(t;\n)+
			\LL\cdot
			(1-q^{-1})\sum_{i=1}^{k}
			[\nabla_\av^{-1}]_{i}^{\phantom{1}}u(t;\n).
			\label{free_eqn}
		\end{align}
		\item 
		For all $\n\in\Z^{k}_{\ge0}$
		such that for some $i\in\{1,2,\ldots,k-1\}$
		one has $n_i=n_{i+1}$, 
		\begin{align}\label{boundary_and_integral_conditions}
			\left([\nabla_\av^{\phantom{1}}]_{i}^{\phantom{1}}
			-q\cdot [\nabla_\av^{\phantom{1}}]_{i+1}^{\phantom{1}}\right)u(t,\n)=0;
			\qquad \qquad
			\left([\nabla_\av^{{-1}}]_{i}^{\phantom{1}}
			-q^{-1}\cdot [\nabla_\av^{{-1}}]_{i+1}^{\phantom{1}}\right)u(t,\n)=0;
		\end{align}
		\item
		For all $\n\in\Z^{k}_{\ge0}$ such that 
		$n_k=0$, $u(t,\n)\equiv 0$ for all $t\ge0$;
		\item
		For all $\n\in\W^{k,N}_{\ge0}$, $u(0,\n)=h_0(\y(\n))$.
	\end{enumerate}
\end{definition}
Note that the boundary conditions 
in \eqref{boundary_and_integral_conditions} coincide with the ones for the $q$-TASEP
\cite{BorodinCorwinSasamoto2012} (and it discrete variants, see
\cite{BorodinCorwin2013discrete}), which 
involve the usual difference operators
$[\nabla]_{i}$ and $[\nabla]_{i+1}$. This is because
$n_i=n_{i+1}$ implies $a_{n_i}=a_{n_{i+1}}$. 
We write the boundary conditions as in \eqref{boundary_and_integral_conditions}
to emphasize their certain similarity with the \integral{} conditions.
\begin{theorem}\label{thm:free}
	If a function $u\colon\R_{\ge0}\times\Z_{\ge0}^{k}\to\R$
	satisfies the free evolution equations
	with $k-1$ boundary and $k-1$
	\integral{} conditions
	(Definition \ref{def:free}), then
	for all $\y\in Y^N_k$, we have 
	$h(t,\y)=u(t,\n(\y))$, where $h$ is the solution
	to the true evolution equations
	(Definition \ref{def:true}) with initial condition
	$h_0(\y)$.	
\end{theorem}
\begin{proof}
	Conditions (3) and (4) of Definition \ref{def:free} 
	directly lead to conditions (2) and (3) 
	of the solution to the true evolution equations
	(Definition \ref{def:true}).
	
	It remains to check that condition
	(1) of Definition \ref{def:true}
	is satisfied by 
	$u(t;\n(\y))$, where $u(t;\n)$ 
	solves
	the free evolution
	equations with $k-1$
	boundary and $k-1$
	\integral{} conditions. We will use \eqref{boundary_and_integral_conditions} to 
	rewrite \eqref{free_eqn} in the form 
	\eqref{true_ODE} (with $\gen^{\text{\rm{}dual}}$
	given by \eqref{dual_operator}). 
	Fix $\n\in\W^{k,N}_{\ge0}$ and let throughout
	the proof $\y=\y(\n)$ and $\n=\n(\y)$, see~\S \ref{sub:notation}.

	First, let us briefly recall 
	(see \cite{BorodinCorwinSasamoto2012})
	how one 
	deals with the summands in 
	\eqref{free_eqn} 
	corresponding to the right jumps.
	Fix any cluster
	of $\n$ of size, say, $c\ge1$, i.e.,
	\begin{align*}
		\n=(n_1\ge \ldots \ge n_{b}>\underbrace{n_{b+1}=\ldots=n_{b+c}}_
		{\text{cluster}}>n_{b+c+1}\ge \ldots\ge n_k\ge0).
	\end{align*}
	Clearly, $c=y_i$, $b=y_N+\ldots+y_{i+1}$, and $n_{b+1}=\ldots=n_{b+c}=i$ 
	for some $i=1,\ldots,N$.
	Combining summands corresponding to $r=b+1,\ldots,b+c$ in the first sum 
	in \eqref{free_eqn} and using the boundary conditions in \eqref{boundary_and_integral_conditions},
	we obtain
	\begin{align*}
		\RR(1-q)\sum_{r=b+1}^{b+c}
		[\nabla_\av^{\phantom{1}}]_{r}^{\phantom{1}}u(t;\n)
		=
		\RR(1-q)\sum_{r=b+1}^{b+c}
		q^{b+c-r}
		[\nabla_\av^{\phantom{1}}]_{b+c}^{\phantom{1}}u(t;\n)
		=\RR(1-q^{c})
		[\nabla_\av^{\phantom{1}}]_{b+c}^{\phantom{1}}u(t;\n).
	\end{align*}
	We readily see that 
	in terms of the $\y$ variables, the above expression
	is equal to 
	\begin{align*}
		\RR(1-q^{y_i})a_i
		\big(
		u\big(t;\n(\y^{i,i-1})\big)-u\big(t;\n(\y)\big)\big),
	\end{align*}
	which is one of the summands in the first sum in \eqref{dual_operator}
	corresponding to the cluster of components of $\n$ which are equal to 
	$i$.

	Now let us explain how one can rewrite the second sum in 
	\eqref{free_eqn} (which corresponds to the left jumps).
	Fix any $j\ge i$ for which $y_j\ge1$. Let us 
	calculate the coefficient by 
	$u\big(t;\n(\y^{j,i})\big)$ in the 
	right-hand side of \eqref{free_eqn}.
	This coefficient can come only from the part of the second sum
	corresponding to the cluster of components of $\n$
	which are equal to $j$.
	Using the \integral{} conditions
	\eqref{boundary_and_integral_conditions}, we 
	can rewrite it as
	(below $b=y_N+\ldots+y_{j+1}$)
	\begin{align*}
		\LL
		(1-q^{-1})\sum_{r=b+1}^{b+y_j}
		[\nabla_\av^{-1}]_{r}^{\phantom{1}}u(t;\n)&=
		\LL
		(1-q^{-1})\sum_{r=b+1}^{b+y_j}q^{r-b-y_j}
		[\nabla_\av^{-1}]_{b+y_j}^{\phantom{1}}u(t;\n)
		\\&=
		\LL
		(1-q^{-y_j})
		[\nabla_\av^{-1}]_{b+y_j}^{\phantom{1}}u(t;\n).
	\end{align*}

	If $j=i$, then we readily see from the above expression
	that the coefficient by 
	$u\big(t;\n(\y^{i,i})\big)=u\big(t;\n(\y)\big)$
	is 
	$\LL a_i^{-1}
	(q^{-y_j}-1)$,
	as it should be according to \eqref{dual_operator}.

	Assume now that $i<j$, and also that $y_{j-1}\ge1$. This means that we can rewrite
	the above expression as
	\begin{align}
		\LL
		(1-q^{-y_{j}})
		[\nabla_\av^{-1}]_{b+y_j}^{\phantom{1}}u(t;\n)
		=\LL
		(q^{-y_{j}}-1)
		\left(a_j^{-1}u(t;\n)+
		[\nabla_\av^{-1}]_{b+y_j+1}^{\phantom{1}}u\big(t;\n(\y^{j,j-1})\big)
		\right).
		\label{free_proof}
	\end{align}
	Indeed, we have simply removed one of the summands from 
	the expression
	$[\nabla_\av^{-1}]_{b+y_j}^{\phantom{1}}u(t;\n)$
	using the definition of $\nabla_\av^{-1}$.
	Now, by \eqref{boundary_and_integral_conditions}, we clearly can write
	the application of 
	$[\nabla_\av^{-1}]_{b+y_j+1}^{\phantom{1}}$	
	to $u\big(t;\n(\y^{j,j-1})\big)$
	as the application of $[\nabla_\av^{-1}]_{b+y_j+y_{j-1}}^{\phantom{1}}$
	times the factor of $q^{-y_{j-1}}$. 
	This observation together with \eqref{free_proof} implies that the coefficient 
	by $u\big(t;\n(\y^{j,j-1})\big)$
	in the right-hand side of \eqref{free_eqn} is equal to 
	$\LL a_{j-1}^{-1}(q^{-y_{j}}-1)q^{-y_{j-1}}$, as it should be 
	by \eqref{dual_operator}. 

	One can check in a similar manner 
	that for any $i<j$,
	the coefficient by 
	$u\big(t;\n(\y^{j,i})\big)$ 
	in the right-hand side of \eqref{free_eqn} 
	is the same as dictated by
	\eqref{dual_operator}.
	This concludes the proof.
\end{proof}

The next statement is a straightforward consequence of Theorems
\ref{thm:true} and \ref{thm:free}:
\begin{corollary}\label{cor:general_moment_formula}
	For $q$-PushASEP started from any fixed or random initial configuration 
	$\x(0)=\x$,
	$\E\left(
	\prod_{i=1}^{k}q^{x_{n_i}(t)+n_i}
	\right)=u(t;\n)$,
	where $u(t;\n)$ solves the free evolution equations
	with $k-1$ boundary and $k-1$ \integral{} conditions
	(Definition \ref{def:free}) with initial data 
	inside the Weyl chamber $\n\in\W^{k,N}_{\ge0}$
	given by 
	$u(0;\n)=\E\left(
	\prod_{i=1}^{k}q^{x_{n_i}(0)+n_i}
	\right)$.
\end{corollary}



\section{Nested contour integral formulas for the $q$-PushASEP} 
\label{sec:nested_contour_integral_formulas_for_the_q_pushasep}

\subsection{Moments: proof of Theorem \ref{thm:moments}} 
\label{sub:moment_formulas_proof_of_theorem_}

Here we will use Corollary \ref{cor:general_moment_formula} to 
prove Theorem~\ref{thm:moments}. 
It suffices to 
check that the 
expression for the moments
of the $q$-PushASEP given by the
right-hand side 
of \eqref{moments} (denote it
by $m(t;\n)$) satisfies conditions (1)--(4) of Definition~\ref{free_eqn}.
Let us verify these conditions:

\medskip

(1) The time derivative in the left-hand side of \eqref{free_eqn} 
affects only the factor $\prod_{j=1}^{k}\frac{\Pi_t(qz_j)}{\Pi_t(z_j)}$
inside the nested integral in $m(t;\n)$, 
which leads to the multiplication 
of the integrand by
\begin{align}\label{1_check}
	\RR(q-1)\sum_{j=1}^{k}z_j+
	\LL(q^{-1}-1)\sum_{j=1}^{k}z_j^{-1}.
\end{align}
Let us check that the application of the 
operators in the right-hand side of \eqref{free_eqn}
also gives the factor \eqref{1_check}.

First, note that for each $j=1,\ldots,k$,
the application of $[\nabla_\av^{\phantom{1}}]_{j}^{\phantom{1}}$
leads to the replacement of 
$\prod_{i=1}^{n_j}\frac{a_i}{a_i-z_j}$ 
by
\begin{align*}
	a_{n_j}\left(\prod_{i=1}^{n_j-1}\frac{a_i}{a_i-z_j}-\prod_{i=1}^{n_j}\frac{a_i}{a_i-z_j}
	\right)
	=-z_j\prod_{i=1}^{n_j}\frac{a_i}{a_i-z_j}.
\end{align*}
We see that we have matched summands involving the parameter $\RR$ in \eqref{1_check}.

Now let us consider the summands in the right-hand side of
\eqref{free_eqn} involving the parameter~$\LL$. 
For simpler notation denote $n=n_j$ and $z=z_j$,
and consider the application of the operator 
$\nabla_\av^{-1}$ to $\prod_{i=1}^{n}\frac{a_i}{a_i-z}$. It is given by
\begin{align*}
	-a_n^{-1}\prod_{i=1}^{n}\frac{a_i}{a_i-z}
	-a_{n-1}^{-1}\prod_{i=1}^{n-1}\frac{a_i}{a_i-z}
	-\ldots-a_{1}^{-1}\frac{a_1}{a_1-z}
	=
	-\prod_{i=1}^{n}\frac{a_i}{a_i-z}
	\cdot
	\sum_{j=1}^{n}a_{j}^{-1}
	\prod_{r=j+1}^{n}\frac{a_r-z}{a_r}.
\end{align*}
Let, by agreement, $a_0=a_{-1}=a_{-2}=\ldots=1$. Let us add to the above sum 
over $j$ more summands corresponding to 
$j$ running from $-\infty$ to $0$, that is, the expression
\begin{align*}
	\sum_{j=-\infty}^{0}
	\prod_{r=j+1}^{n}\frac{a_r-z}{a_r}=
	\frac{1}{z}\prod_{r=1}^{n}\frac{a_r-z}{a_r}.
\end{align*}
In view of the nested contour integration
in \eqref{moments}, we see that these additional summands 
(when multiplied by 
$\prod_{i=1}^{n}\frac{a_i}{a_i-z}$)
do not introduce 
any residues in $z$. 
Thus, modulo the contour integration, we can rewrite 
the application of 
$\nabla_\av^{-1}$ to $\prod_{i=1}^{n}\frac{a_i}{a_i-z}$
as
\begin{align*}
	-\prod_{i=1}^{n}\frac{a_i}{a_i-z}
	\left(
	\frac{1}{z}
	\prod_{r=1}^{n}\frac{a_r-z}{a_r}+
	\sum_{j=1}^{n}a_{j}^{-1}
	\prod_{r=j+1}^{n}\frac{a_r-z}{a_r}
	\right).
\end{align*}
To finish the check of (1) 
by matching summands involving the parameter $\LL$ in \eqref{1_check},
it suffices to establish
the following lemma:
\begin{lemma}
	We have
	\begin{align}\label{1_check_lemma}
		\frac{1}{z}
		\prod_{r=1}^{n}\frac{a_r-z}{a_r}+
		\sum_{j=1}^{n}a_{j}^{-1}
		\prod_{r=j+1}^{n}\frac{a_r-z}{a_r}=\frac{1}{z}.
	\end{align}
\end{lemma}
\begin{proof}
	Denote by $S_n$ the left-hand side of 
	\eqref{1_check_lemma}. Then one can readily see that
	\begin{align*}
		S_0=1/z;
		\qquad
		S_{n-1}\frac{a_n-z}{a_n}=S_n-a_n^{-1},\quad n\ge1,
	\end{align*}
	which implies the claim.
\end{proof}

(2) The argument is almost the same for the boundary and 
the \integral{} conditions.
As we have seen in the above check of (1),
the (boundary condition) 
operator $[\nabla_\av^{\phantom{1}}]_{i}^{\phantom{1}}
-q\cdot [\nabla_\av^{\phantom{1}}]_{i+1}^{\phantom{1}}$
applied to $m(t;\n)$, multiplies 
the integrand
by 
$-(z_{i}-qz_{i+1})$. The 
(\integral{} condition) operator
$[\nabla_\av^{{-1}}]_{i}^{\phantom{1}}
-q^{-1}\cdot [\nabla_\av^{{-1}}]_{i+1}^{\phantom{1}}$
leads to the multiplication of the integrand in 
$m(t;\n)$ by
\begin{align*}
	-\left(\frac{1}{z_i}-q^{-1}\frac{1}{z_{i+1}}\right)=
	\frac{z_i-qz_{i+1}}{q z_iz_{i+1}}.
\end{align*}
In both cases, the factor $z_i-qz_{i+1}$ cancels one of the denominators
in $\prod_{1\le A<B\le k}\frac{z_A-z_B}{z_A-qz_B}$.
This allows us to deform (without encountering any poles) the $z_i$ and $z_{i+1}$
contours so that they coincide. 
Since $n_i=n_{i+1}$, this means that
we may write 
both
$\left([\nabla_\av^{\phantom{1}}]_{i}^{\phantom{1}}
-q\cdot [\nabla_\av^{\phantom{1}}]_{i+1}^{\phantom{1}}\right)m(t;\n)$
and $\left([\nabla_\av^{{-1}}]_{i}^{\phantom{1}}
-q^{-1}\cdot [\nabla_\av^{{-1}}]_{i+1}^{\phantom{1}}\right)m(t;\n)$ in the form
\begin{align*}
	\int\int(z_i-z_{i+1})G(z_i,z_{i+1})dz_{i}dz_{i+1}
\end{align*}
for a suitable function $G(z_i,z_{i+1})$ 
involving integration in all variables except
$z_i$ and $z_{i+1}$.
The function $G$ (in both cases)
is symmetric in $z_i,z_{i+1}$, which 
implies that the above integral is 
identically zero.
Thus, the second condition in Definition \ref{def:free} is also checked.

\smallskip

(3) To check the third condition, 
observe that
if $n_k=0$, then there is no pole $z_k=1$ 
in the integral over $z_k$ in \eqref{moments}. Thus, 
the nested integral vanishes.

\smallskip

(4) Because for the step initial
condition $x_i(0)=-i$, the left-hand side 
of \eqref{moments} is identically one. 
We thus need to show that $m(0;\n)\equiv1$.
This follows from the residue calculus. Expanding the $z_1$ contour
to infinity, one encounters only the pole at $z_1=0$ (clearly, $z_1=\infty$ is not a pole
because of the factors $a_i/(a_i-z_1)$). The residue at $z_1=0$
is equal to $-q^{-(k-1)}$. After having expanded the $z_1$
contour, the remaining integral is the same as in 
\eqref{moments} but in $k-1$ variables. 
Thus, repeating this proceedure, we see that the fourth condition is also verified.

\medskip

By virtue of Corollary \ref{cor:general_moment_formula},
this completes the proof of Theorem \ref{thm:moments}.\qed


\subsection{Discussion of the Fredholm determinantal formula (Conjecture \ref{claim:Fredholm})} 
\label{sub:discussion_of_the_fredholm_determinantal_formula}

Assume that $\LL>0$.
Let us first discuss the growth of the moments of the $q$-PushASEP. 
\begin{lemma}\label{lemma:growth_of_moments}
	For any $k\ge1$ and $\n\in\W^{k,N}_{\ge0}$, 
	\begin{align*}
		\E\left(\prod_{i=1}^{k}q^{x_{n_i}(t)+n_i}\right)
		\ge \mathrm{const}\cdot e^{\LL a_1^{-1}t \cdot e^{k\ln(1/q)}}.
	\end{align*}
	Here $\mathrm{const}$ is some positive constant,
	and $\x(t)$ is the $q$-PushASEP started from an arbitrary 
	(non-random) initial configuration $\x(0)=\x$.
\end{lemma}
\begin{proof}
	Clearly, 
	\begin{align*}
		\prod_{i=1}^{k}q^{x_{n_i}(t)+n_i}\ge
		q^{k(x_1(t)+1)}=\mathrm{const}\cdot q^{k(\xi-\eta)},
	\end{align*}
	where $\xi$ and $\eta$ are independent Poisson 
	random variables with parameters $\RR a_1 t$ and $\LL a_1^{-1}t$, 
	respectively (cf. the end of \S \ref{sub:definition_of_the_process}). 
	The constant in front accounts for the initial
	condition $x_1(0)$.
	We have
	\begin{align*}
		\E q^{k(\xi-\eta)}=e^{\RR a_1 t(q^{k}-1)+\LL a_1^{-1}t(q^{-k}-1)},
	\end{align*}
	which yields the claim.
\end{proof}

Let us now explain how one could \emph{formally} 
obtain the Fredholm determinant (Conjecture \ref{claim:Fredholm}) 
from the moment formulas of Theorem \ref{thm:moments} that
were proved in \S \ref{sec:true_evolution_equations}--\ref{sec:free_q_pushasep}.
Using the $q$-Binomial theorem,
write the $q$-Laplace transform as 
\begin{align*}
	\E\left[\frac{1}{(\zeta q^{x_n(t)+n};q)_{\infty}}\right]=
	\E\left[\sum_{k=0}^{\infty}
	\frac{\zeta^{k}q^{k(x_n(t)+n)}}{(q;q)_{k}}
	\right].
\end{align*}
This identity is rigorous.
Next, let us interchange the expectation
and the summation. By Lemma \ref{lemma:growth_of_moments}, 
we get a \emph{divergent} series
(of course this is because the interchange of expectation and summation is not justifiable):
\begin{align}\label{divergent}
	\sum_{k=0}^{\infty}
	\frac{\zeta^{k}\E(q^{k(x_n(t)+n)})}{(q;q)_{k}}.
\end{align}
However, plugging nested contour integral expressions
for the moments $\E(q^{k(x_n(t)+n)})$ 
afforded by Theorem \ref{thm:moments},
it is possible to formally write 
\eqref{divergent}
as a Fredholm determinant. 
A general scheme for doing this is explained in \S3.1 of \cite{BorodinCorwinSasamoto2012} and was 
initially developed in \S3.2 of \cite{BorodinCorwin2011Macdonald}. 
It amounts to deforming (and accounting for residues coming from this deformation) 
the nested contours 
in \eqref{moments}
so that they all become a small circle 
around $z=1$.
This is a rigorous operation, see
\cite[Prop. 3.2]{BorodinCorwinSasamoto2012}.
Then one should reorder summands in \eqref{divergent}, and also use the 
Mellin-Barnes summation formula. These two latter operations 
may not be done in a rigorous way in our situation.\footnote{When $\LL=0$,
i.e., for the $q$-TASEP, all operations are valid, see
\cite{BorodinCorwin2011Macdonald}
and \cite{BorodinCorwinSasamoto2012}.}
However, applied to the divergent series \eqref{divergent}, these steps yield a 
valid Fredholm determinantal expression of Conjecture \ref{claim:Fredholm}.

It is possible that Conjecture \ref{claim:Fredholm} 
(which we have formally argued for above) can be rigorously proved with 
the help of the algebraic framework of Macdonald processes \cite{BorodinCorwin2011Macdonald}. 
Namely, one may be able to show
(in a manner similar to 
\cite{BCGS2013}, \cite{BorodinCorwinFerrariVeto2013})
that
identity \eqref{Fred_det_identity} 
is a specialization of an algebraic identity
which in turn
can be established 
without running into convergence issues.
Then \eqref{Fred_det_identity} arises 
for certain particular values of parameters.
Another possible way of resolving the convergence issues
is to pass to a suitable discrete-time regularization,
cf. the discussion in \S \ref{sub:fredholm_determinant}.



\appendix

\section{Dynamics on two-dimensional interlacing arrays} 
\label{sec:two_dimensional_dynamics}

Here we briefly explain how the $q$-PushASEP 
arises as a one-dimensional 
marginal of a certain two-dimensional 
stochastic Markov dynamics on interlacing arrays of particles. 
This two-dimensional dynamics may be 
constructed
as an interpolation
between the ``push-block'' dynamics of \cite[\S2.3.3]{BorodinCorwin2011Macdonald}
(see also Dynamics 1 in \cite[\S5.5]{BorodinPetrov2013NN}), and 
the $q$-version of the dynamics driven by row insertion RSK algorithm
(Dynamics 8 in \cite[\S8.2.1]{BorodinPetrov2013NN}).\footnote{There is no unique way of
defining a dynamics on 
two-dimensional interlacing arrays
with these properties. For instance, the 
``push-block'' dynamics may be replaced by the 
dynamics coming from the 
$q$-version of the Robinson-Schensted column insertion algorithm
introduced in \cite{OConnellPei2012}.
See also
\cite{BorodinPetrov2013NN}
for more examples and 
a general discussion. }
Note that the latter dynamics has to be reflected, i.e., the particles
under this dynamics
must jump to the left instead of jumping to the right.
Let us now proceed to the definition of the 
two-dimensional dynamics.

The state space of the two-dimensional dynamics is the set of triangular arrays 
of interlacing particles which have integer coordinates (see Fig.~\ref{fig:la_interlacing} for an example):
\begin{align*}
	\boldsymbol\la=\{\la^{(k)}_{j}\in\Z,\; 1\le j\le k\le N\colon \la^{(k)}_{j}\le\la^{(k-1)}_{j-1}\le\la^{(k)}_{j-1}\}.
\end{align*}
Each particle $\la^{(k)}_{j}$ can jump either to the right or to the left by one.
\begin{figure}[htbp]
	\begin{center}
		\begin{tabular}{cc}
			\includegraphics[width=171pt]{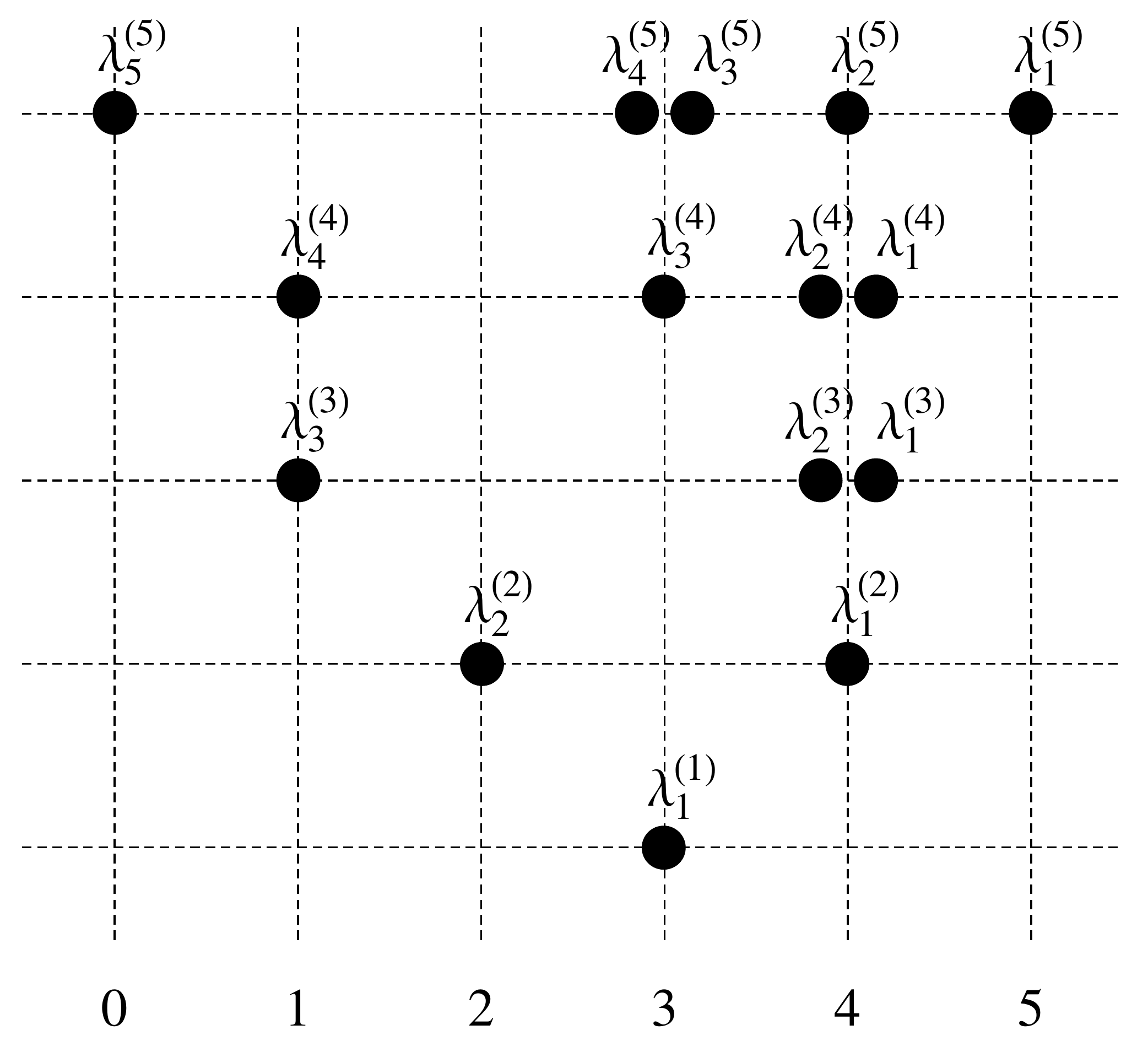}
			&
			\includegraphics[width=171pt]{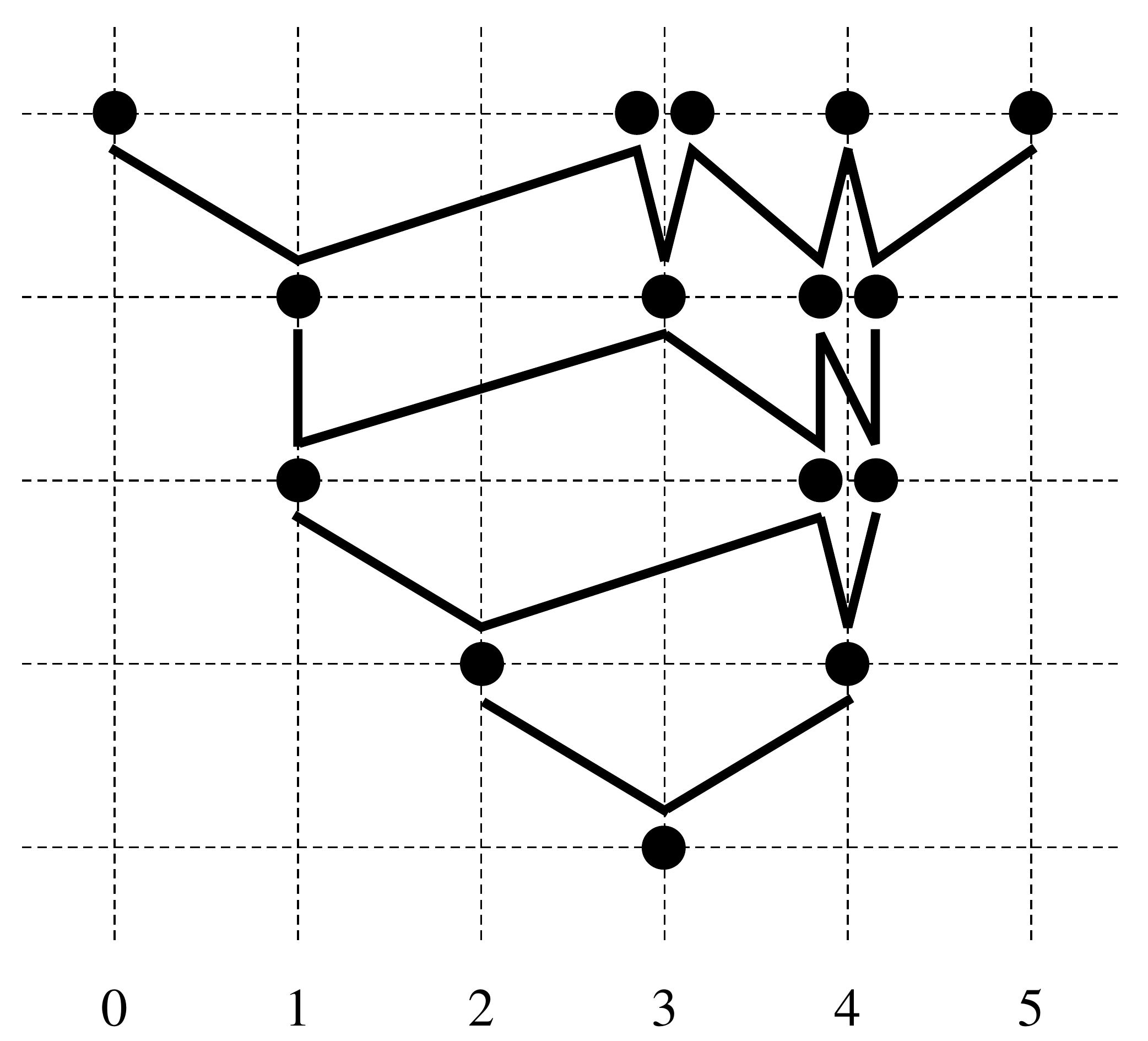}
		\end{tabular}
	\end{center}
  	\caption{Particle configuration $\boldsymbol\la$
  	and a visualization of the interlacing property.}
  	\label{fig:la_interlacing}
\end{figure}

The right jumps are described as follows. 
Each particle $\la^{(k)}_{j}$ has an independent exponential clock
with rate 
\begin{align*}
	\RR a_k \frac{(1-q^{\la^{(k-1)}_{j-1}-\la^{(k)}_{j}})(1-q^{\la^{(k)}_{j}-\la^{(k)}_{j+1}+1})}
	{1-q^{\la^{(k)}_{j}-\la^{(k-1)}_{j}+1}}.
\end{align*}
When the clock of $\la^{(k)}_{j}$
rings, the particle jumps to the right by one. If this jump of $\la^{(k)}_{j}$
would break the interlacing with upper particles, i.e., if $\la^{(k)}_{j}=\la^{(k+1)}_{j}=\ldots=\la^{(k+m)}_{j}$
(for some $m\ge1$), 
then all the particles $\la^{(k+1)}_{j},\ldots,\la^{(k+m)}_{j}$
are instantaneously pushed to the right by one.\footnote{
This mechanism of instantaneous pushes is built into the jump 
rates. Indeed, if the interlacing is broken, then
the higher particles have infinite
jump rates due to vanishing denominator.
Moreover, if the jump of some
$\la^{(k)}_{j}$ would break the interlacing with lower particles, then 
the rate assigned to this jump is equal to zero.}

The left jumps are different. Only the leftmost particles $\la^{(k)}_{k}$
can independently jump to the left by one. At level $k$ of the 
array the independent jumps of left particles happen at rate $\LL a_k^{-1}$.
When any particle $\la^{(k-1)}_{j}$ moves to the left by one
(independently or due to a push), it instantaneously 
forces one of its two immediate upper neighbors, 
$\la^{(k)}_{j+1}$ or $\la^{(k)}_{j}$, to move to the left by one
with probabilities $\ell$ and $1-\ell$, respectively, where
\begin{align*}
	\ell=q^{\la^{(k-1)}_{j}-\la^{(k)}_{j+1}}
	\frac
	{1-q^{\la^{(k)}_{j+1}-\la^{(k-1)}_{j+1}}}
	{1-q^{\la^{(k-1)}_{j}-\la^{(k-1)}_{j+1}}}
\end{align*}
(here $\la^{(k-1)}_{j}$ denotes the position of the particle \emph{before} the move).

In the description of the dynamics, all factors of the form $(1-q^{\cdots})$ 
having nonexistent indices are set to be equal to one.
One can readily see that the leftmost particles under this two-sided dynamics
on two-dimensional interlacing arrays
marginally evolve as a \emph{Markov} process. 
In the shifted coordinates
$x_n(t):=\la^{(n)}_{n}(t)-n$, where $n=1,\ldots,N$, 
the evolution 
of the particles
is governed by our $q$-PushASEP.

\medskip

The fixed-time distributions of the two-dimensional dynamics
$\boldsymbol\la(t)$ described above are probability measures 
on interlacing arrays.
Let the initial configuration be the 
\emph{densely packed} one, i.e.,
$\la^{(k)}_{j}(0)=0$ for all $1\le j\le k\le N$.
This configuration corresponds to the step initial
condition for the $q$-PushASEP.

After time $t\ge0$,
the distribution of 
$\boldsymbol\la(t)$
generalizes the
(one-sided) Macdonald processes of 
\cite{BorodinCorwin2011Macdonald}, \cite{BCGS2013}.
The second Macdonald parameter which
is usually denoted by 
$t$ is set to zero
(so that there is no notational conflict with the time parameter);
such Macdonald processes 
are also referred to as the $q$-Whittaker processes.

Put $a_i\equiv1$ for simplicity.
If $\LL$ is zero, then 
$\boldsymbol\la(t)$ is distributed according to 
\begin{align}\label{Macd_process}
	\prob\big(\boldsymbol\la(t)\big)=\frac{1}{Z}
	P_{\la^{(1)}}(1)P_{\la^{(2)}/\la^{(1)}}(1)
	\ldots 
	P_{\la^{(N)}/\la^{(N-1)}}(1)Q_{\la^{(N)}}(\rho_{\RR t}),
\end{align}
where each $\la^{(k)}=(\la^{(k)}_{1}\ge \ldots\ge \la^{(k)}_{k})\in\Z^{k}$ 
is an ordered collection of nonnegative integers, 
$P$ and $Q$ are the (ordinary and skew) Macdonald symmetric functions \cite{Macdonald1995},
and $\rho_{\RR t}$ is the so-call Plancherel specialization 
of $Q_{\la^{(N)}}$, e.g., see \cite[\S2.2.1]{BorodinCorwin2011Macdonald}.
The Plancherel specialization may be defined, e.g., in terms of the 
generating function for the \emph{one-row} Macdonald $Q$ functions
(i.e., functions indexed by ordered $k$-tuples of integers with $k=1$):
\begin{align}\label{one_sided_Plancherel}
	\sum_{n\ge0}Q_{(n)}(\rho_t)u^{n}=e^{t u}.
\end{align}

On the other hand, for $\RR=0$, the distribution of $-\boldsymbol\la(t)$
(this simply means negating all components 
of the interlacing array)
is described by the Macdonald process 
\eqref{Macd_process} (with $\RR t$ replaced by $\LL t$ in the 
Plancherel specialization of $Q_{-\la^{(N)}}$).

In the general case when $\LL$ and $\RR$ are both positive, 
we expect that the distribution of $\boldsymbol\la(t)$ (started from the
packed initial configuration) is given by a certain
\emph{two-sided} version of a Macdonald process.
This two-sided version should necessarily have the form
\begin{align}\label{two_sided_Macd}
	\prob\big(\boldsymbol\la(t)\big)=\frac{1}{Z}
	P_{\la^{(1)}}(1)P_{\la^{(2)}/\la^{(1)}}(1)
	\ldots 
	P_{\la^{(N)}/\la^{(N-1)}}(1)
	\mathcal{M}_{N}^{(\RR t;\LL t)}(\la^{(N)})
\end{align}
for a suitable nonnegative function 
$\mathcal{M}_{N}^{(\RR t;\LL t)}$ on the $N$th floor
(cf. \eqref{Macd_process}). Note that here
the coordinates $\la^{(k)}_j$ can be positive or negative 
(but still must interlace).

Indeed, the product of the $P$ functions,
$P_{\la^{(1)}}(1)P_{\la^{(2)}/\la^{(1)}}(1)
\ldots 
P_{\la^{(N)}/\la^{(N-1)}}(1)$,
corresponds to a certain \emph{Gibbs} property
of Macdonald processes 
(see \cite{BorodinPetrov2013NN} for more detail)
which is preserved by 
both the dynamics with $\LL=0$ or $\RR=0$,
and thus also by the dynamics with general positive
$\RR$ and $\LL$ (this is because the Markov generator
of the latter process is a linear combination 
of the two ``pure'' right and left generators).

When $N=1$, the measure \eqref{two_sided_Macd}
is simply the convolution of the two ``pure'' one-sided measures
(note that $P_{\la^{(1)}}(1)=1$), and so
the generating function for 
$\mathcal{M}_{1}^{(\RR t;\LL t)}$
takes the form (cf. \eqref{Pi_t})
\begin{align}\label{two_sided_Plancherel}
	\sum_{n\in\Z}\mathcal{M}_{1}^{(\RR t;\LL t)}(n)
	u^{n}=e^{t(\RR u+\LL u^{-1})}.
\end{align}

Note that in the one-sided case, 
the one-row functions
$Q_{(n)}$ generate the algebra of symmetric functions
to which all the $Q_{\la}$'s (with $\la$ having nonnegative parts) belong.
Thus, identity \eqref{one_sided_Plancherel} defines $Q_{\la}(\rho_t)$ for all $\la$,
and one can proceed to the definition of the one-sided Macdonald processes.
In the two-sided case, 
it is not clear what algebraic
structures are responsible for the passage from 
$\mathcal{M}_{1}^{(\RR t;\LL t)}(n)$
(viewed as 
one-row functions
$Q_{(n)}
(\rho_{\RR t;\,\LL t}^{\text{two-sided}})$)
to the functions $Q_\la$ with $\la$ general.
Therefore, at this point we are left to view 
\eqref{two_sided_Macd} as a \emph{definition}
of the two-sided Plancherel specialization of the general
Macdonald symmetric functions
$Q_{\la^{(N)}}(\rho_{\RR t;\,\LL t}^{\text{two-sided}})
:=\mathcal{M}_{N}^{(\RR t;\LL t)}(\la^{(N)})$
corresponding to not necessarily one-row $\la$'s.
We do not further develop the theory of two-sided Macdonald processes
in the present paper, but note that
the desire to understand the distribution of the 
two-sided dynamics on two-dimensional interlacing
integer arrays \eqref{two_sided_Macd},
as well as the question of proving 
Conjecture \ref{claim:Fredholm},
provide some motivation for these 
objects.\footnote{The $q=0$ version of the two-sided Macdonald
processes (i.e., the two-sided Schur processes) 
was introduced and investigated in 
\cite{Borodin2010Schur}.}


\section{Formal scaling limit as $q\nearrow1$} 
\label{sec:formal_scaling_limit_as_}

Consider the scaling of the two-dimensional dynamics described by 
\cite[Thm. 4.1.21]{BorodinCorwin2011Macdonald}:
\begin{align*}&
	q=e^{-\varepsilon},\qquad
	t=\varepsilon^{-2}\tau,\qquad
	a_k=e^{-\varepsilon\mathsf{a}_k},& k=1,\ldots,N;
	\\&	
	\la^{(k)}_{j}=
	C(\varepsilon;\tau)-(k+1-2j)\varepsilon^{-1}\log\varepsilon
	+G^{(k)}_{j}\varepsilon^{-1},&
	k=1,\ldots,N,\ j=1,\ldots,k.
\end{align*}
Here $\tau>0$ is the scaled time, 
$C(\varepsilon;\tau)$ represents the global shift of the 
coordinate system,
and $(\mathsf{a}_1,\ldots,\mathsf{a}_N)$
are the scaled values of the $a_j$'s.
In the one-sided setting, the Macdonald processes 
\eqref{Macd_process}
converge under this scaling with $C(\varepsilon;\tau)=\varepsilon^{-2}\tau$ 
to Whittaker processes introduced in \cite{Oconnell2009_Toda},
see also \cite[Ch. 4]{BorodinCorwin2011Macdonald}.

As explained in 
\cite[\S 4.1 and \S 5.2]{BorodinCorwin2011Macdonald} and
\cite[\S 8.4]{BorodinPetrov2013NN},
the $q$-TASEP and the $q$-PushTASEP (i.e., the ``pure'' dynamics
corresponding to $\LL=0$ or $\RR=0$)
under this scaling with $C(\varepsilon;\tau)=+\varepsilon^{-2}\tau$
or $C(\varepsilon;\tau)=-\varepsilon^{-2}\tau$, respectively, correspond to 
stochastic differential equations
(SDEs) which describe evolution of the 
hierarchy of the 
free energies of the O'Connell--Yor 
semi-discrete directed polymer 
\cite{OConnellYor2001}, \cite{Oconnell2009_Toda}.\footnote{These are our quantities
$G^{(k)}_{k}$ in the description of the scaling.}
These free energies may also be represented 
as logarithms of 
solutions to the semi-discrete stochastic heat equation
\begin{align}\label{SHE}
	du_j(t)=u_{j-1}(t)-u_j(t)+u_j(t)dB_j(t),\qquad
	j=1,\ldots,N;\qquad
	\qquad
	u(0,N)=\delta_{1N},
\end{align}
where $B_1,\ldots,B_N$ are independent standard Brownian motions (possibly with linear drifts).

Let us now discuss the formal scaling limit of the
two-sided
($q$-PushASEP) evolution, i.e., with $\RR,\LL>0$.
Let us scale the $\RR$ and $\LL$ parameters around 1:
\begin{align*}
	\RR=e^{-\varepsilon \mathsf{r}},\qquad
	\LL=e^{-\varepsilon \mathsf{l}},
\end{align*}
where $\mathsf{r}, \mathsf{l}\in\R$ are the scaled values.
Moreover, one should take the global shift 
$C(\varepsilon;\tau)$ to be zero (one should think that the 
shifts $\pm\varepsilon^{-2}\tau$ corresponding to the 
``pure'' right and left dynamics compensate each other).

We will focus only on the 
leftmost particles $\la^{(k)}_{k}$, the whole array
can be considered in a similar way.
The limiting SDEs for the quantities 
$G^{(k)}_{k}$ look as (with the agreement that $G^{(0)}_{0}\equiv 0$)
\begin{align}\label{qToda}
	dG^{(k)}_{k}=\sqrt 2 \cdot dW_{k}
	+
	\left(
	-2\mathsf{a}_k+\mathsf{l}
	-\mathsf{r}-e^{G^{(k)}_{k}-G^{(k-1)}_{k-1}}
	\right)d\tau,\qquad k=1,\ldots,N.
\end{align}
Here $W_1,\ldots,W_N$ 
are independent standard driftless Brownian motions.
\begin{remark}
	The $G^{(k)}_{k}$'s 
	satisfying \eqref{qToda}
	can also be formally interpreted 
	as logarithms of
	solutions to the semi-discrete stochastic heat equation
	\eqref{SHE}.
	The terms $(-2\mathsf{a}_k+\mathsf{l}
	-\mathsf{r})$ are absorbed into
	drifts of the Brownian motions $B_1,\ldots,B_N$
	in \eqref{SHE}.
\end{remark}
Calculations leading to \eqref{qToda} are analogous to what is done in
\cite[\S 5.4.4]{BorodinCorwin2011Macdonald} and
\cite[\S 8.4.4]{BorodinPetrov2013NN}.
First, note that our scaling dictates
\begin{align}\label{G_la_scaling}
	G^{(k)}_{k}(\tau+d\tau)-G^{(k)}_{k}(\tau)=\frac{\la^{(k)}_{k}(\tau+\varepsilon^{-2}d\tau)
	-\la^{(k)}_{k}(\tau)}{\varepsilon^{-1}}.
\end{align}
Right jumps of the particle $\la^{(k)}_{k}$
occur with probability 
\begin{align}\label{SHE_1}
	\RR a_k(1-q^{\la^{(k-1)}_{k-1}-\la^{(k)}_{k}})
	=1-\varepsilon(\mathsf{a}_k+\mathsf{r}
	+ e^{G^{(k)}_{k}-G^{(k-1)}_{k-1}})+O(\varepsilon^{2}).
\end{align}
Left jumps happen at rate 
\begin{align}\label{SHE_2}
	\LL a_k^{-1}=e^{-\varepsilon(\mathsf{l}-\mathsf{a}_k)}
	=1-(\mathsf{l}-\mathsf{a}_k)\varepsilon+O(\varepsilon^{2}),
\end{align}
and, moreover, the particle $\la^{(k)}_{k}$ is pushed to the left 
by $\la^{(k-1)}_{k-1}$ with probability $\varepsilon e^{G^{(k)}_{k}-G^{(k-1)}_{k-1}}+O(\varepsilon^{2})$.
One should multiply this probability by the change in the position of
$\la^{(k-1)}_{k-1}$ during time interval $t=\varepsilon^{-2}\tau$, this yields 
\begin{align}\label{SHE_3}
	\Big(\varepsilon e^{G^{(k)}_{k}-G^{(k-1)}_{k-1}}+O(\varepsilon^{2})\Big)
	\Big(
	\varepsilon^{-1}
	\big(
	G^{(k-1)}_{k-1}(\tau+d\tau)
	-G^{(k-1)}_{k-1}(\tau)
	\big)
	\Big)=O(1).
\end{align}
The constant factors in \eqref{SHE_1} and
\eqref{SHE_2} give rise to the change
in the position of $\la^{(k)}_{k}$
(during time interval $\varepsilon^{-2}d\tau$)
equal to the difference of two independent 
Poisson random variables with mean $\varepsilon^{-2}d\tau$.
In view of \eqref{G_la_scaling}, these summands correspond to 
the differential of the Brownian motion $\sqrt 2 \cdot dW_{k}(\tau)$.
The summands of order $\varepsilon$ in 
\eqref{SHE_1}--\eqref{SHE_2} give rise to constant terms.
The constant term \eqref{SHE_3} is multiplied by $\frac{1}{\varepsilon^{-1}}$, 
and thus vanishes. 


\providecommand{\bysame}{\leavevmode\hbox to3em{\hrulefill}\thinspace}
\providecommand{\MR}{\relax\ifhmode\unskip\space\fi MR }
\providecommand{\MRhref}[2]{%
  \href{http://www.ams.org/mathscinet-getitem?mr=#1}{#2}
}
\providecommand{\href}[2]{#2}


\begin{thebibliography}{10}

\bibitem{alimohammadi1999two}
M.~Alimohammadi, V.~Karimipour, and M.~Khorrami, \emph{{A two-parametric family
  of asymmetric exclusion processes and its exact solution}}, Journal of
  statistical physics \textbf{97} (1999), no.~1-2, 373--394,
  arXiv:cond-mat/9805155.

\bibitem{Balasz_Komjathy_Seppalainen}
M.~Bal\'asz, J.~Komj\'athy, and T.~Sepp{\"a}l{\"a}inen, \emph{{Microscopic
  concavity and fluctuation bounds in a class of deposition processes}}, Ann.
  Inst. H. Poincar\'e B \textbf{48} (2012), 151--187.

\bibitem{Bethe1931}
H.~Bethe, \emph{{Zur Theorie der Metalle. I. Eigenwerte und Eigenfunktionen der
  linearen Atomkette. (On the theory of metals. I. Eigenvalues and
  eigenfunctions of the linear atom chain)}}, Zeitschrift fur Physik
  \textbf{71} (1931), 205--226.

\bibitem{Borodin2010Schur}
A.~Borodin, \emph{{Schur dynamics of the Schur processes}}, Advances in
  Mathematics \textbf{228} (2011), no.~4, 2268--2291, arXiv:1001.3442
  [math.CO].

\bibitem{BorodinCorwin2013discrete}
A.~Borodin and I.~Corwin, \emph{{Discrete time q-TASEPs}}, Intern. Math.
  Research Notices (2013), arXiv:1305.2972 [math.PR], doi: 10.1093/imrn/rnt206.

\bibitem{BorodinCorwin2011Macdonald}
\bysame, \emph{Macdonald processes}, Prob. Theory Rel. Fields \textbf{158}
  (2014), 225--400, arXiv:1111.4408 [math.PR].

\bibitem{BorodinCorwinFerrariVeto2013}
A.~Borodin, I.~Corwin, P.~Ferrari, and B.~Veto, \emph{{Height fluctuations for
  the stationary KPZ equation}},  (2014), arXiv:1407.6977 [math.PR].

\bibitem{BCGS2013}
A.~Borodin, I.~Corwin, V.~Gorin, and S.~Shakirov, \emph{{Observables of
  Macdonald processes}},  (2013), arXiv:1306.0659 [math.PR].

\bibitem{BorodinCorwinPetrovSasamoto2013}
A.~Borodin, I.~Corwin, L.~Petrov, and T.~Sasamoto, \emph{{Spectral theory for
  the q-Boson particle system}}, Compositio Mathematica \textbf{151} (2015),
  no.~1, 1--67, arXiv:1308.3475 [math-ph].

\bibitem{BorodinCorwinSasamoto2012}
A.~Borodin, I.~Corwin, and T.~Sasamoto, \emph{{From duality to determinants for
  q-TASEP and ASEP}}, Ann. Probab. \textbf{42} (2014), no.~6, 2314--2382,
  arXiv:1207.5035 [math.PR].

\bibitem{BorFerr08push}
A.~Borodin and P.~Ferrari, \emph{{Large time asymptotics of growth models on
  space-like paths I: PushASEP}}, Electron. J. Probab. \textbf{13} (2008),
  1380--1418, arXiv:0707.2813 [math-ph].

\bibitem{BorodinPetrov2013NN}
A.~Borodin and L.~Petrov, \emph{{Nearest neighbor Markov dynamics on Macdonald
  processes}},  (2013), arXiv:1305.5501 [math.PR], to appear in Adv. Math.

\bibitem{Calabrese_LeDoussal_Rosso}
P.~Calabrese, P.~Le~Doussal, and A.~Rosso, \emph{{Free-energy distribution of
  the directed polymer at high temperature}}, Euro. Phys. Lett. \textbf{90}
  (2010), no.~2, 20002.

\bibitem{ODEs}
E.A. Coddington and N.~Levinson, \emph{{Theory of Ordinary Differential
  Equations}}, McGraw Hill, 1955.

\bibitem{Dotsenko}
V.~Dotsenko, \emph{{Replica Bethe ansatz derivation of the Tracy-Widom
  distribution of the free energy fluctuations in one-dimensional directed
  polymers}}, Journal of Statistical Mechanics: Theory and Experiment (2010),
  no.~07, P07010, arXiv:1004.4455 [cond-mat.dis-nn].

\bibitem{Ethier1986}
S.N. Ethier and T.G. Kurtz, \emph{Markov processes: {C}haracterization and
  convergence}, Wiley-Interscience, New York, 1986.

\bibitem{Liggett1999}
T.~Liggett, \emph{{Stochastic Interacting Systems: Contact, Voter and Exclusion
  Processes}}, {Grundlehren de mathematischen Wissenschaften}, vol. 324,
  Springer, 1999.

\bibitem{Liggett1985}
\bysame, \emph{{Interacting Particle Systems}}, Springer-Verlag, Berlin, 2005.

\bibitem{Macdonald1995}
I.G. Macdonald, \emph{Symmetric functions and {H}all polynomials}, 2nd ed.,
  Oxford University Press, 1995.

\bibitem{MatveevPetrov2014}
K.~Matveev and L.~Petrov, \emph{{$q$-deformed Robinson--Schensted--Knuth
  correspondences and discrete time $q$-TASEPs}},  (2014), In preparation.

\bibitem{Oconnell2009_Toda}
N.~O'Connell, \emph{{Directed polymers and the quantum Toda lattice}}, Ann.
  Probab. \textbf{40} (2012), no.~2, 437--458, arXiv:0910.0069 [math.PR].

\bibitem{OConnellPei2012}
N.~O'Connell and Y.~Pei, \emph{{A q-weighted version of the Robinson-Schensted
  algorithm}}, Electron. J. Probab. \textbf{18} (2013), no.~95, 1--25,
  arXiv:1212.6716 [math.CO].

\bibitem{OConnellYor2001}
N.~O'Connell and M.~Yor, \emph{{Brownian analogues of Burke's theorem}},
  Stochastic Processes and their Applications \textbf{96} (2001), no.~2,
  285--304.

\bibitem{Povolotsky2013}
A.~Povolotsky, \emph{{On integrability of zero-range chipping models with
  factorized steady state}}, J. Phys. A \textbf{46} (2013), 465205.

\bibitem{Povolotsky_Mendes_2006}
A.~Povolotsky and J.F.F. Mendes, \emph{{Bethe ansatz solution of discrete time
  stochastic processes with fully parallel update}}, Journal of Statistical
  Physics \textbf{123} (2006), no.~1, 125--166, arXiv:cond-mat/0411558
  [cond-mat.stat-mech].

\bibitem{SasamotoWadati1998}
T.~Sasamoto and M.~Wadati, \emph{{Exact results for one-dimensional totally
  asymmetric diffusion models}}, J. Phys. A \textbf{31} (1998), 6057--6071.

\bibitem{Spitzer1970}
F.~Spitzer, \emph{{Interaction of Markov processes}}, Adv. Math. \textbf{5}
  (1970), no.~2, 246--290.

\end{thebibliography}
\end{document}